\pdfoutput=1
\documentclass[11pt,reqno]{amsart}
\usepackage{amsmath, amssymb, amscd, amsfonts, amsthm, mathrsfs, mathtools, url, pinlabel, verbatim, flafter, comment, setspace}
\usepackage[centering,letterpaper,dvips]{geometry}
\usepackage{color, multirow, dcpic, latexsym, graphicx, epstopdf, subcaption}
\usepackage{subdepth}
\usepackage[colorlinks=true,citecolor=blue,linkcolor=blue]{hyperref}
\usepackage{tikz-cd}
\usepackage{amsthm}
\usepackage{bm}
\usepackage{todonotes}
\usepackage{graphicx}

\newtheorem {theorem}{Theorem}

\newtheorem {lemma}[theorem]{Lemma}
\newtheorem {corollary}[theorem]{Corollary}

\theoremstyle{definition}

\newtheorem {definition}[theorem]{Definition}
\newtheorem {question}[theorem]{Question}

\newtheorem *{remark}{Remark}

\newcommand{\hfkhat}{\widehat{\mathit{HFK}}}
\newcommand{\cfkhat}{\widehat{\mathit{CFK}}}
\newcommand{\cfahat}{\widehat{\mathit{CFA}}}
\newcommand{\cfdhat}{\widehat{\mathit{CFD}}}
\newcommand{\cfkminus}{\mathit{CFK}^{-}}

\title{Cable knots are not thin}

\author{Subhankar Dey}
\address{Department of Mathematics, IIT Palakkad}
\email{subhankardey@iitpkd.ac.in}

\date{}

\begin{document}
\maketitle

\begin{abstract}
Using the Bordered Floer theory of Lipshitz-Ozsv\'ath-Thurston we prove that the $(p,q)$-cables of any non-trivial knots are not Heegaard Floer homologically thin. Using the proof and a theorem of Zemke, we find a larger set of satellite knots which is a proper superset of the set of all cable knots, having the same property.
\end{abstract}

\section{Introduction} 
  In his seminal work \cite{thurston82, thurston86}, Thurston showed that a knot, based on the geometry on its complement, is either one of three types: torus, satellite, or hyperbolic. Apart from that classification, there is a family of knots that are easy to describe diagrammatically, namely, alternating knots, which admit projections onto generic planes, that `alternate' between under-passes and over-passes. It was proved by Menasco in \cite{menasco84} that an alternating knot is either a torus or a hyperbolic knot.
\begin{theorem}[\cite{menasco84}] 
\label{thm:mer_lemma}
If L is a non-split prime alternating link, and if $S \subset S^3 \smallsetminus L$ is a closed incompressible surface, then $S$ contains a circle which is isotopic in $S^3 \smallsetminus L$ to a meridian of $L$.
\end{theorem}
The above theorem and the fact that the exterior of a satellite knot contains an incompressible torus implies that prime alternating knots are not satellite. Menasco's proof of Theorem \ref{thm:mer_lemma} makes direct use of alternating knot diagrams. More recently, Ozsv\'ath-Szab\'o defined a larger class of knots, called \emph{quasi-alternating knots} (see \cite[Definition 3.1]{osbranch05}).
\begin{definition} Let $\mathcal{Q}$ denote the smallest set of links such that 
 \begin{itemize}
     \item the unknot is a member of $\mathcal{Q}$. 
     \item if $L$ is a member of $\mathcal{Q}$, then there exists a projection of $L$ and a crossing $c$ in that projection such that 
     \begin{enumerate}
     \item both smoothings of $L$ at $c$ (see Figure \ref{skein}), $L_0$ and $L_{\infty}$ are in $\mathcal{Q}$,
     \item $\det(L) = \det(L_0) + \det(L_{\infty})$.
     \end{enumerate}
 \end{itemize}
 \end{definition}
 
 \begin{figure}
    \centering
    \includegraphics[width=8cm]{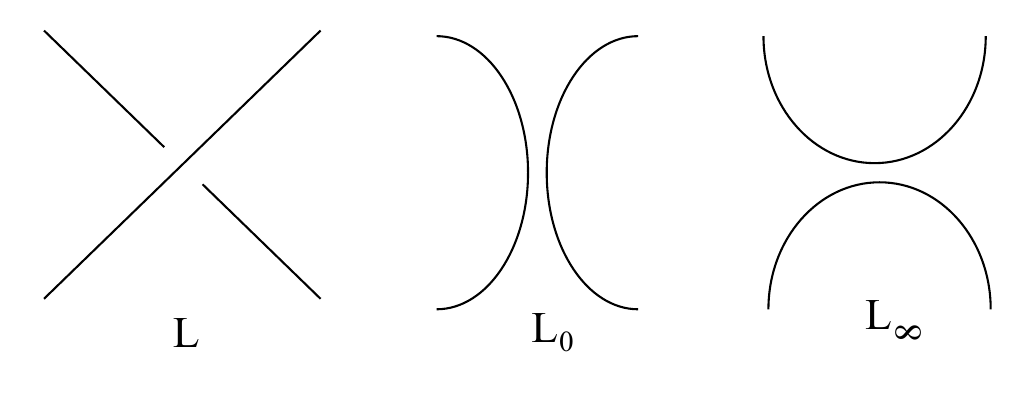}
    \caption{smoothings of a crossing}
    \label{skein}
\end{figure} 
 
The knot Floer homology of knots belonging to this set exhibit the same kind of characteristics as the knot Floer homology of alternating knots, hence the name quasi-alternating. Specifically, one of the characteristics is that quasi-alternating knots are \emph{Heegaard Floer homologically thin}, which is to say that the knot Floer homology of a quasi-alternating knot is supported in gradings where the difference between Alexander and Maslov gradings is fixed \cite[Theorem 1.2]{mo07} and they are completely determined by the signature of the knot and its Alexander polynomial. Throughout the paper, we will refer to this as \textit{Floer-thin}. Similar notion of thinness exists in Khovanov homology as well. Ozsv\'ath-Szab\'o proved that double branched covers of quasi-alternating knots are $L$-spaces, i.e. $\widehat{\textit{HF}}(\Sigma(K))\cong\mathbb{Z}_2^{\det(K)}$ \cite[Proposition 3.3]{osbranch05}, which was utilized by Gordon-Lidman in \cite[Theorem 1.3; Theorem 1.4]{gl14} to show that cable knots are not quasi-alternating. Recently Boyer-Gordon-Hu in \cite[Theorem 2.15]{bgh21} proved a folklore conjecture that prime satellite knots are not quasi-alternating, which can be thought of as a generalization of the mentioned result of Menasco from alternating to quasi-alternating knots. In particular they prove the following:

\begin{theorem}[Theorem 2.15 in \cite{bgh21}] A prime $\mathbb{Z}/2$ Khovanov-thin link is either hyperbolic or a $(2,m)$-torus link. In particular, this holds for prime, quasi-alternating links.
    
\end{theorem}

In \cite{greene10}, Greene showed that the set of all Floer-thin knots properly contains the set of all quasi-alternating knots. Therefore it is natural to ask if the same is true in the setting of knot Floer homology:

\begin{question}
\label{conj:qanotthin}
Are there prime satellite knots which are Floer-thin?
\end{question}

 The methods in the mentioned papers do not generalize for Floer-thin knots. Ideally, one would like to prove an analog of Theorem \ref{thm:mer_lemma} for Floer-thin knots in the context of Heegaard Floer homology. Recently Petkova-Wong in \cite{pw20} showed that satelite knots with twisted Mazur patterns are not Floer-thin. However, in \cite{zib22} Zibrowius found the first example of a prime satellite knot which is Floer-thin, which naturally raises the curiosity concerning Question \ref{conj:qanotthin} if there are more such satellite knots.


 Showing that a non-trivial (satellite) knot with trivial Alexander polynomial cannot be Floer-thin is quite straight-forward. Recall that for a knot $K$ with (symmetric) Alexander polynomial $\Delta_K(t) = \sum_{i=0}^{g} a_i (t^i + t^{-i})$,
\begin{gather*}
 a_i = \chi(\hfkhat(S^3,K,i))  = \sum_{\alpha \in \hfkhat(S^3,K,i)} (-1)^{m_\alpha}  
\end{gather*}
where $m_{\alpha}$ is the Maslov grading of $\alpha$. For Floer-thin knots, all elements in a fixed Alexander grading of its knot Floer homology have the same Maslov grading (see \cite[Theorem 1.2]{mo07}), in other words, $\hfkhat(S^3,K,i) = \mathbb{Z}_2^{|a_i|}$. This and the fact that knot Floer homology detects the Seifert genus of a knot (i.e. $\hfkhat(S^3,K,g_3(K))\not= 0$) imply that $K$ must be the unknot. Hence satellite knots with trivial Alexander polynomial, for example, Whitehead doubles, are not quasi-alternating. For a satellite knot with an arbitrary pattern and having non-trivial Alexander polynomial, this argument cannot be used. For non-trivial cable knots, we can prove the following.
\begin{theorem}
\label{thm:main} The $(p,q)$-cable of any non-trivial knot $K$ is not Floer-thin.

\end{theorem}
This theorem recovers a part of result of Boyer-Gordon-Hu from \cite{bgh21}:
\begin{corollary}
If $K$ is non-trivial, $K_{p,q}$ is not quasi-alternating. 
\end{corollary}

In \cite[Theorem 1.1]{zemke19}, Zemke proves that if there is a ribbon concordance $C$ from $K_0$ to $K_1$, then there is an injection $F_C : \widehat{HFK}(K_0) \rightarrow \widehat{HFK}(K_1)$, which preserves both Alexander and Maslov gradings. Combined with Zemke's result, the proof of Theorem \ref{thm:main} implies the following corollary.

\begin{corollary}
\label{morethin}
If there is a ribbon concordance from a cable knot $K_{p,q}$ to another knot $K'$, then $K'$ is not Floer-thin. In particular, $K'$ is not quasi-alternating.  

\end{corollary}

Miyazaki in \cite{miyazaki98} proved that non-trivial band sums between knots are ribbon concordant to the trivial band sum between them i.e. connected sum of those knots. One can start with a cable knot $K_{p,q}$ and place a number of unknots inside a small solid 3-ball inside the solid torus, such that they are unlinked to the pattern, and then join them by some non-trivial bands between them and the $T_{p,q}$ pattern sitting already inside that neighborhood, such that the bands stays inside the neighborhood. Now if one considers the resulting knot inside the solid torus as the pattern and take the companion as $K$, then the resulting satellite knot $P(K)$ is ribbon concordant to $K_{p,q}$ by Miyazaki's result. Hence by Corollary \ref{morethin}, $P(K)$ is not Floer-thin. Figure \ref{band} is such an example of a pattern $P$, which is ribbon concordant to $T_{5,3}$. 
\begin{figure}
    \centering
    \includegraphics[width=8cm]{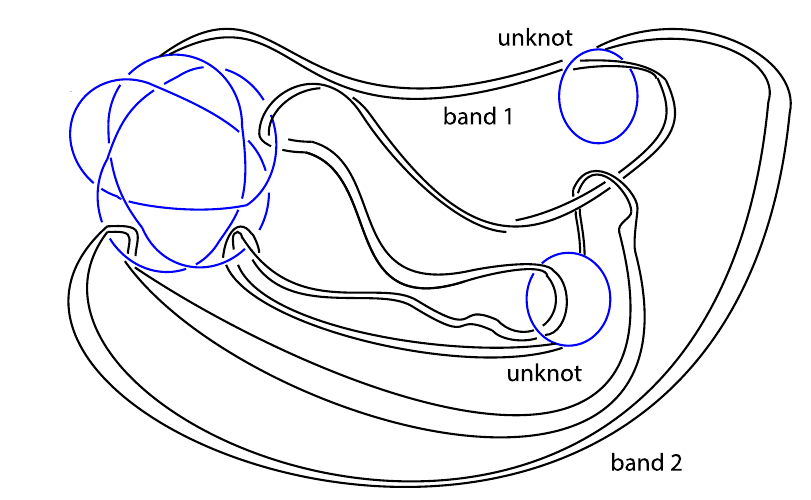}
    \caption{A `not-thin' schematic pattern: obtained by ribbon summing $T(5,3)$ with two unknots, both living inside a small ball in $S^1 \times D^2$}
    \label{band}
\end{figure} 
\begin{remark}
 An alternative way to prove that a knot is not Floer-thin could be by comparing  $\tau$-invariant, a knot concordance invariant defined by Ozsv\'ath-Sz\'abo in \cite{os4genus03} and Rasmussen in \cite{rasmussen03}, and the signature of the knot, since for thin knots $\tau(K)$ equals negative of the half of its signature. In recent years, there has been much work studying effects of cabling on various knot invariants. Shinohara in \cite{shinohara71} studied signature and satellite operation. Hedden studied effect of cabling for knots on knot Floer homology  in \cite{heddencabling05}, \cite{heddencabling09}. Binns and the author studied the cabling operations for links in \cite{cablinglink}. Petkova in \cite{petkova13} and later Hom in \cite{hom11} studied $\tau$ invariant and cabling operation, while Chen in \cite{chen16} studied $\Upsilon$ invariant and cabling and Chakraborty in \cite{c19} studied Legendrian knot invariant $\hat{\theta}$ and cabling.
 In particular for $\tau$ of a cable knot, Hedden provided an inequality in \cite{heddencabling09}, later improved by Hom in \cite{hom11}, in terms of the $\tau$ of the companion knots. Using that and Shinohara's result \cite[Theorem 9]{shinohara71} about signature of cable knots, one can prove that a certain class of cable knots are not thin (eg. most of the iterated torus knots). We thank Abhishek Mallick for pointing this out to the author. 

 It is also interesting to compare the methods applied in this paper with the cabling formula of Hanselman-Watson in \cite{cablingimmersed}.
 
\end{remark}

\begin{subsection}*{Main Idea and Organization}
In order to prove Theorem \ref{thm:main}, we will use bordered Floer homology package of Lipshitz-Ozsv\'ath-Thurston in \cite{lot08}, which is tailor made to study satellite knots. The main idea of the proof is to produce two non-zero elements in the knot Floer homology of a given cable knot, using the splicing Theorem \ref{splice}, such that the difference between their Alexander grading is different than the difference between their Maslov grading. Thus by \cite[Theorem 1.2]{mo07}, they are not Floer-thin. To do that we will look at $\cfahat(D^2 \times S^1, T_{p,q})$ and $\cfdhat(X_K)$ simultaneously to find two elements from both modules such that their box tensor product is non-zero in the homology. The splicing theorem allows us to find two such elements, living inside the knot Floer chain complex of the cable knot. Then we show that those elements are non-zero in the knot Floer homology of the cable knot and calculate their grading to get the desired result. 

In  Section 2, we briefly discuss the algebraic structure of bordered Floer homology. In Section 3, we use the aforementioned splicing theorem to find two elements in the knot Floer homology of any cable knot such that the difference between their Alexander grading and that of their Maslov grading are not same, to deduce that it cannot be Floer-thin. 
    
    \end{subsection}

\subsection*{Acknowledgements}
I am grateful to my PhD advisor, \c{C}a\u{g}atay Kutluhan, for his encouragement, guidance and inputs during the preparation of the first draft. I would like to thank Matt Hedden whose initial suggestion to look into bordered Floer homology led to the result and also for his feedback on an earlier draft. Special thanks to William Menasco and Xingru Zhang for several valuable discussions during the preparation of the first draft, and Abhishek Mallick for suggesting this problem and a number of helpful discussions. I would also like to forward my gratitude to the anonymous referees for being patient in pointing out flaws in the draft, and helping to improve the draft immensely. I would also like to thank Holt Bodish for pointing out a mistake in an earlier version of the paper and to Claudius Zibrowius for providing invaluable feedback on the draft. An individual research grant of the DFG, project no. 505125645, supported my work.

\section{Background on bordered Floer homology}
We start by describing briefly some features of bordered Floer homology which we will be using to prove the main theorem. We will only be interested in the case when a compact manifold has torus boundary. A bordered 3-manifold (with torus boundary) is a compact manifold with (torus) boundary, along with a diffeomorphism $ \phi : T^2 \rightarrow \partial Y$, up to isotopy fixing a neighborhood of a point, cf. \cite{hl16} and \cite [Definition 1.4] {lotnotes}. 

There are several versions of pairing theorem \cite[Theorem 1.3] {lot08} in bordered Floer homology that come in handy to study Heegaard Floer holomogy of a closed 3-manifold obtained by splicing two manifolds with torus boundaries or a manifold cut along a torus (see \cite{hl16},\cite{hanselman17}). We will be interested in the following splicing theorem:
\begin{theorem}[Theorem 11.19 in \cite{lot08}]
\label{splice}
Let $Y_1$ and $Y_2$ be 3-manifolds with torus boundary and $K_1 \subset Y_1$ be a knot in $Y_1$. If gluing $Y_1$ and $Y_2$ along their boundaries
$ \partial Y_1 = \partial Y_2 = F$ produces a null-homologous knot $K \subset Y_1 \cup_F Y_2$, then there is a homotopy equivalence of $\mathbb{Z}$-filtered chain complexes:

\begin{equation}
    \cfkhat(Y,K) \simeq \cfahat(Y_1,K_1)\boxtimes \cfdhat(Y_2)
\end{equation}

which respects the gradings.
\end{theorem}

To make sense about the modules $\cfdhat$, $\cfahat$ mentioned above, we start by reviewing the setting of bordered Floer Homology.

For a compact manifold $Y$ with torus boundary, a bordered Heegaard diagram is a tuple $\mathcal{H} = (\Sigma_g, \alpha_1^a, \alpha_2^a,\alpha_1^c,\alpha_2^c, \cdots, \alpha_{g-1}^c,\beta_1,\cdots,\beta_g,z)$ such that 

\begin{itemize}
    \item $\Sigma_g$ is a compact, oriented surface of genus $g$ with one boundary component,
    \item $\bm{\beta} = \{\beta_1, \cdots, \beta_g\}$ is a $g$-tuple of pairwise disjoint circles in the interior of $\Sigma_g$,
    \item $\bm{\alpha^c} = \{\alpha_1^c, \cdots, \alpha_{g-1}^c\}$ is a ($g-1$) tuple of pairwise disjoint circles in the interior of $\Sigma_g$,
    \item $\bm{\alpha^a} = \{\alpha_1^a,\alpha_2^a\}$ is a 2-tuple of pairwise disjoint arcs in $\Sigma_g$ with boundary in $\partial \Sigma_g$,
    \item $z$ is a base point in $\partial \Sigma_g \smallsetminus  \bm{\alpha^a} $,
    \item $\bm{\alpha^a} \cap \bm{\alpha^c} = \emptyset$,
    \item both $\Sigma_g \smallsetminus (\bm{\alpha^a} \cup \bm{\alpha^c})$ and $\Sigma_g \smallsetminus \bm{\beta}$ are connected.
    
\end{itemize}

To provide an example, given a knot $K \subset S^3$, we can find a bordered Heegaard diagram of $S^3 \smallsetminus K$, by first starting with a \textit{doubly pointed} Heegaard diagram of $(S^3,K)$, $(\Sigma,\bm{\alpha},\bm{\beta},w,z)$ i.e. $ (\Sigma, \bm{\alpha},\bm{\beta},z)$ is a Heegaard diagram of $S^3$. Recall that the basepoints $w, z$, are two points on the Heegaard surface such that joining $w$ to $z$ in the complement of $\bm{\beta}$ curves and joining $z$ to $w$ in the complement of $\bm{\alpha}$ curves and subsequently pushing those arcs into the $\bm{\alpha}$ and $\bm{\beta}$ handlebodies, respectively, produces $K$. Given such a diagram, first one can stabilize the Heegaard diagram by doing a surgery at two points near the base points $w,z$. Then draw a longitude $\beta_g$ of the knot that goes over that newly attached handle (see \cite[Fig.4]{osknot04}) and thus gets a meridian of the knot, living on the attached 2-handle, call it $\alpha_g$.  Hence $(\Sigma', \bm{\alpha} \cup \alpha_g,\bm{\beta} \cup \beta_g, z,w )$ is a stabilized doubly pointed knot diagram. Take $\lambda$ to be a closed curve, parallel to $\beta_g$, intersecting $\alpha_g$ at one point, say $p$. Let $D$ be a small neighborhood disk around $p$.  Then the complement of $int(D)$ specifies a bordered Heegaard diagram that represents a bordered Heegaard diagram of $S^3 \smallsetminus K$, specifically $(\Sigma' \setminus int (D^2),\bm{\alpha} \cup \alpha'_g \cup \lambda', \bm{\beta} \cup \beta_g, z')$, where $\alpha'_g = \alpha_g \setminus int (D) = \alpha_g \cap \Sigma' , \lambda' = \lambda \setminus int(D) = \lambda \cap \Sigma'$ and $z'$ lies on $\partial D$, away from the endpoints of the $\alpha$-arcs.

To get a bordered Heegaard diagram of an $n$-framed knot complement, instead of considering $\lambda$, one needs to take $\lambda_n$, which can be obtained by winding $\lambda$ around $\alpha_g, n$ times. After that operation, one of the $\alpha$-arcs would be $ \lambda'_n = \lambda_n \setminus int(D) =  \lambda_n \cap \Sigma'$, instead of $\lambda'$ mentioned earlier (see \cite[Section 2.6]{lee18}).

One can define a \textit{doubly pointed} bordered Heegaard diagram of $K$ in a manifold $Y$, by finding such a bordered Heegaard diagram of $Y$, adding an extra basepoint in which would represent $K \subset Y$, just as a doubly pointed diagram for a knot in a closed 3-manifold. 

Given a bordered Heegaard diagram $\mathcal{H} = (\Sigma, \bm{\alpha},\bm{\beta},z) $ of a manifold, a \emph{generator} $x = \{x_1,x_2,\cdots,x_g\} \subset \bm{\alpha} \cap \bm{\beta} $ is such that exactly one point can lie on each $\bm{\beta}$ circle, exactly one point can lie on each $\bm{\alpha}$-circle and at most one point can lie on each $\bm{\alpha}$-arcs, where $\bm{\alpha} = \bm{\alpha^a} \cup \bm{\alpha^c}$. Let $\mathfrak{G}(\mathcal{H})$ be the set of all such generators.

\vspace{.3in}

\noindent With the above settings in place, we now discuss the algebraic preliminaries regarding bordered Floer homology that will eventually lead to explaining the ingredients of the splicing Theorem \ref{splice}.

Let $\mathcal{Z}$ denote the boundary of $D$ in a bordered Heegaard diagram. We consider the \textit{pointed matched circle} ($\mathcal{Z},\textbf{a},M, z)$  with $4k$ marked points 
$\textbf{a} = \{a_1, \cdots, a_{2k}, a'_1,\cdots,a'_{2k}\}$
and $M(a_i) = a'_i, i = 1,2,\cdots,2k $, a pairing of the points and $z$ being a base point on $\mathcal{Z}$. For $Y$, a compact 3-manifold with torus boundary, we only focus on the case where $k=1$ and we call $\textbf{a} = \{a_0,a_1,a_2,a_3\}$ such that $M(a_0) = a_2,M(a_1) = a_3$ and the points $a_0,a_1,a_2,a_3$ are labeled on $\mathcal{Z}$ in a clockwise direction.

A pointed matched circle represents a compact surface with one boundary component in the following way: consider $\mathcal{Z} \times [0,1] $ and then add bands i.e. two-dimensional one-handles with feet on the pairs of matched points on the circle on $\mathcal{Z} \times \{0\}$ and then attach a disk to the new boundary component created after adding the bands. If we cap off the remaining boundary component with a disk, we call that closed surface of genus-$k$ $F(\mathcal{Z})$. In particular, when $k=1$, $F(\mathcal{Z})$ is a torus and we are interested in this case only since all the non-closed manifolds that appear in this paper have torus boundary. Let $\alpha_1^a, \alpha_2^a$ denote the arcs (both lying on the surface), which are the cores of the 1-handles, from $a_0$ to $a_2$ and from $a_1$ to $a_3$, respectively (See Figure 3). 

\begin{figure}[h!]
    \centering
    \includegraphics[width=6cm]{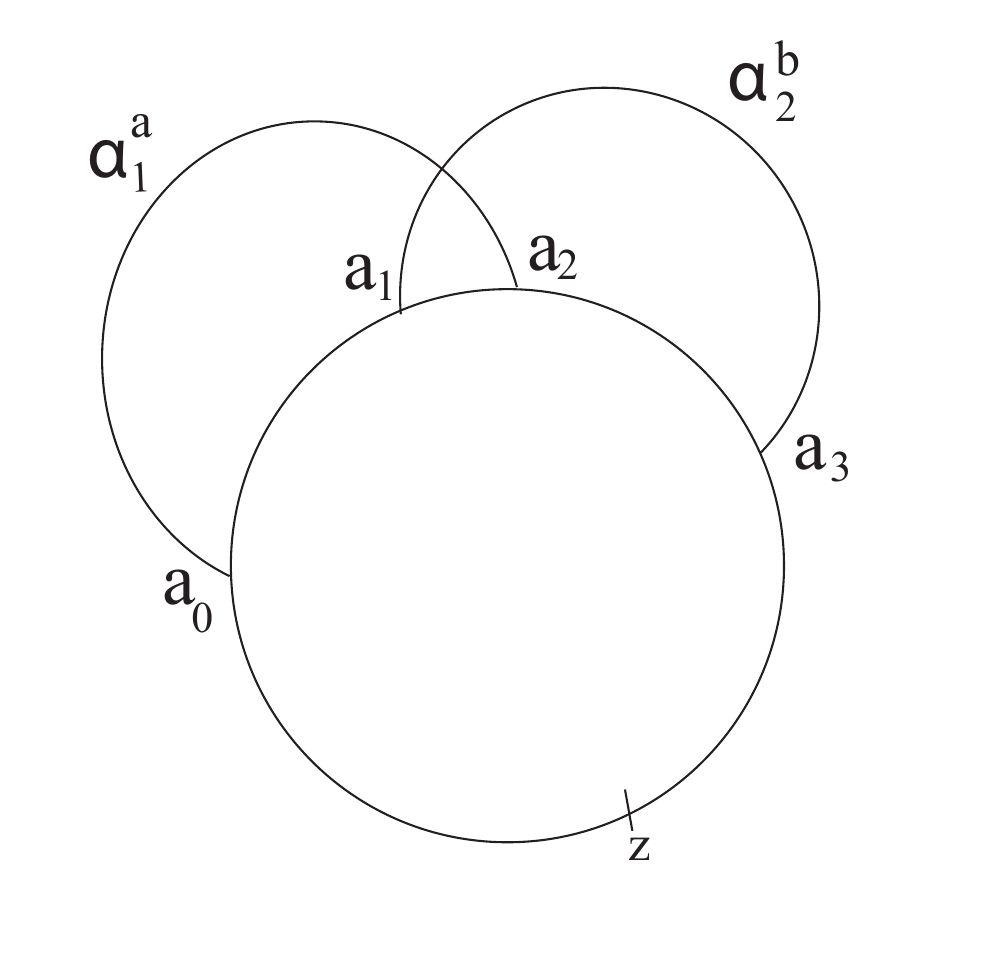}
    \caption{A pointed matched circle}
    \label{circle}
\end{figure} 

For the case in hand i.e. when $F(\mathcal{Z})$ represents a torus, $\mathcal{A}(\mathcal{Z})$ is an unital algebra over $\mathbb{F}$ with six \textit{`Reeb'} elements $\rho_1,\rho_2,\rho_3,\rho_{12},\rho_{23},\rho_{123}$ (in this case, one can think $\rho_1,\rho_2,\rho_3$ to be the arcs in $\mathcal{Z}$ between $a_0,a_1$, between $a_1,a_2$ and between $a_2,a_3$, respectively) and two \textit{idempotents} $\iota_0,\iota_1$ such that $\iota_0 + \iota_1 = 1$ and these generators have these non-zero relations:
    \[\rho_1 = \iota_0 \rho_1 = \rho_1 \iota_1 \quad  \rho_2 = \iota_1 \rho_2 = \rho_2 \iota_0 \quad \rho_3 = \iota_0 \rho_3 = \rho_3 \iota_1\] 
    \[\iota_0 \rho_{12} = \rho_{12} \iota_0 = \rho_{12} \quad \iota_1 \rho_{23} = \rho_{23} \iota_1 = \rho_{23} \quad \iota_0 \rho_{123} = \rho_{123} \iota_1 = \rho_{123}\]
    \[\rho_1 \rho_2 = \rho_{12} \quad \rho_2 \rho_3 = \rho_{23} \quad \rho_{12}\rho_3 = \rho_1 \rho_{23} = \rho_{123} \]

Through out the paper, if not mentioned otherwise, we denote $\mathcal{A}(\mathcal{Z})$ by $\mathcal{A}$. For more in depth discussion on $\mathcal{A}(\mathcal{Z})$ for $k > 1$, see Chapter 3 of \cite{lot08} or Chapter 1.4 \cite{lotnotes} for a short exposition. 

In \cite{lot08}, Lipshitz-Ozsv\'ath-Thurston associate to a bordered 3-manifold $(Y,\phi: \partial Y \rightarrow F(\mathcal{Z}))$, $\cfdhat(Y)$ and $\cfahat(Y)$, which are right $\mathcal{A}_{\infty}$ module over $\mathcal{A}(\mathcal{Z})$ and left $dg$ module over $\mathcal{A}(-\mathcal{Z})$ module, respectively. We concern ourselves with the modules $\cfahat(Y,K)$ and $\cfdhat(Y)$ for this paper. 

A vector space $M$ over $\mathbb{F}$ is said to have a \emph{(right) type A structure} over $\mathcal{A}$ or is said to be a right $\mathcal{A}_{\infty}$-module over $\mathcal{A}$ (where $\mathcal{A}$ is a graded, unital algebra) if $M$ is equipped with a right action of $\mathcal{I}$ (the set of all idempotents in $\mathcal{A}$, in our case $\mathcal{I} = \{\iota_0,\iota_1\}$), such that $M = M\iota_0 \oplus M\iota_1 $, as a vector space, and multiplication maps
\begin{align*}
 m_{k+1} : M\otimes \mathcal{A}^{\otimes k} \rightarrow M  , \quad k \geq 0  
\end{align*}
satisfying the $\mathcal{A}_{\infty}$ relations
\begin{align*}
    0 = \sum_{i=0}^{k} m_{k-i+1}(m_{i+1}(x, a_1, a_2, \cdots, a_i), a_{i+1}, \cdots, a_k) \\ +  \sum_{i=1}^{k-1} m_k(x, a_1, \cdots, a_{i-1}, a_ia_{i+1}, a _{i+2}, \cdots, a_k)
\end{align*}
and the unital conditions
\begin{eqnarray*}
m_2(x,1) &=& x \\
m_k(x, \cdots, 1, \cdots ) &=& 0, \quad k >2.
\end{eqnarray*}
We say that M is $bounded$ if $m_k = 0$ for sufficiently large $k$.

Now $\cfahat(Y,K)$ is a $\mathbb{F}$-module generated by $\mathfrak{G}(\mathcal{H})$, where the right action by $\mathcal{I}$ is defined by  
\begin{align*}
    \bm{x} \cdot \iota_0 = 
    \begin{cases} 
    \bm{x} &\quad\text{if $\bm{x}$ \emph{does} occupy the arc $\alpha^a_1$} \\
    0 &\quad\text{otherwise}
\end{cases} \\
  \bm{x} \cdot\iota_1 = 
  \begin{cases}
  \bm{x} &\quad\text{if $\bm{x}$ \emph{does} occupy the arc $\alpha^a_2$} \\
  0 &\quad\text{otherwise}
  \end{cases}
\end{align*}
The right $\mathcal{A}_{\infty}$-module structure on $\cfahat(Y,K)$ is determined by the multiplication maps
\begin{align*}
    m_{k+1} : \cfahat(Y,K) \otimes \mathcal{A}^{\otimes k} \rightarrow \cfahat(Y,K)
\end{align*}
defined by
\begin{align*}
    & m_{k+1}(\bm{x}, \rho_{i_1}, \cdots,\rho_{i_k}) = \sum_{\bm{y} \in \mathfrak{G}(\mathcal{H})} \sum\limits_{\substack{\{B\in \pi_2(\bm{x},\bm{y})\;|\;\it{ind}(B,(\rho_{i_1},\cdots,\rho_{i_k}))= 1,\\n_w(B) = 0\}}} \#(\mathcal{M}^{B}(\bm{x},\bm{y},(\rho_{i_1},\cdots, \rho_{i_k}))) \bm{y}, \\
    & m_2(\bm{x},1) = \bm{x}, \\
   & m_{k+1}(\bm{x},\cdots,1,\cdots) = 0 \quad \rm{for}\; k > 1,
\end{align*}
where $\mathcal{M}(\bm{x},\bm{y},(\rho_{i_1},\cdots,\rho_{i_k}))$ is as described in \cite[Chapter 2] {lotnotes}. For $\bm{x}, \bm{y} \in \mathfrak{G}(\mathcal{H})$, both $g$-tuples,  let $S$ be a smooth surface with boundary and punctures on its boundary and label $g$ of these punctures $-$, another $g$ punctures $+$, and the remaining punctures $e$. Then $\mathcal{M}(\bm{x},\bm{y},(\rho_{i_1},\cdots,\rho_{i_k}))$ consists of holomorphic embeddings
\begin{gather*}
    u : (S, \partial S) \rightarrow (\Sigma \smallsetminus \{z\}  \times [0,1] \times \mathbb{R}, (\bm{\alpha}\times \{1\} \times \mathbb{R}) \cup (\bm{\beta} \times \{0\} \times \mathbb{R}))
\end{gather*}
such that 
\begin{itemize}
    \item at the $-$ punctures, $u$ is asymptotic to $\bm{x} \times [0,1] \times \{ -\infty\}$, 
    \item at the $+$ punctures, $u$ is asymptotic to $\bm{y} \times [0,1] \times \{ \infty\}$,
    \item at the punctures labeled $e$, $u$ is asymptotic to the chords $\rho_i \times (1,t_i) \in \partial \Sigma \times \{1\} \times \mathbb{R}$, and $t_{i_1}<t_{i_2}<\cdots<t_{i_k}$.
\end{itemize}

If $\widetilde{\mathcal{M}}(\bm{x},\bm{y},(\rho_1,\rho_2,\cdots,\rho_n))$ denotes the moduli space of $J$-holomorphic maps satisfying the above properties, then $\mathcal{M}(\bm{x},\bm{y},(\rho_1,\rho_2,\cdots,\rho_n)) =  {\widetilde{\mathcal{M}}(\bm{x},\bm{y},(\rho_1,\rho_2,\cdots,\rho_n))} / \mathbb{R} $ (where $\mathbb{R}$ denotes the translation action in the image). Given $B \in \pi_2(\bm{x},\bm{y}), \mathcal{M}^{B}(\bm{x},\bm{y},(\rho_1,\cdots,\rho_n))$ is the set of all such $J$-holomorphic maps which has the same homology as that of $B$. See \cite[Chapter 5]{lot08} for details. The above defined family of multiplication maps $\{m_i\}$ counts the number of points in the described moduli space. \\

Now we describe the other module structure mentioned above. A vector space $N$ over $\mathbb{F}$ is said to have a \emph{(left) type D structure} over $\mathcal{A}$, if $N$ is equipped with a left action of $\mathcal{I}$ (set of all idempotent elements $\{\iota_0, \iota_1 \}$ in $\mathcal{A}$) such that $N = \iota_0 N \oplus \iota_1 N$, as a vector space, and a map
\begin{align*}
    \delta_1 : N \rightarrow \mathcal{A} \otimes N 
\end{align*}
satisfying the type D condition:
\begin{align*}
    (\mu \otimes id_N) \circ (id_{\mathcal{A}} \otimes \delta_1) \circ \delta_1 = 0
\end{align*}
where $\mu: \mathcal{A}\otimes \mathcal{A} \rightarrow \mathcal{A}$ is the multiplication map in $\mathcal{A}$. Also, we inductively define maps 
\begin{gather*}
    \delta_k : N \rightarrow \mathcal{A}^{\otimes k} \otimes N
\end{gather*}
such that $\delta_0 = id_N, \delta_i = (id_{\mathcal{A}^{\otimes (i-1)}} \otimes \delta_1) \circ \delta_{i-1}$. We say $N$ is $bounded$ if $\delta_k$ is zero for sufficiently large $k$. \\

$\cfdhat(Y)$ is a $\mathbb{F}$-vector space generated by $\mathfrak{G}(\mathcal{H})$, where the left action by $\mathcal{I}$ is defined by 

\begin{align*}
    \iota_0 \cdot \bm{x} = 
    \begin{cases} 
    \bm{x} &\quad\text{if $\bm{x}$ \emph{does not}  occupy the arc $\alpha^a_1$} \\
    0 &\quad\text{otherwise}
\end{cases} \\
  \iota_1 \cdot \bm{x} = 
  \begin{cases}
  \bm{x} &\quad\text{if $\bm{x}$ \emph{does not} occupy the arc $\alpha^a_2$} \\
  0 &\quad\text{otherwise}
  \end{cases}
\end{align*}
and the left $\mathcal{A}$-module structure on $\cfdhat(Y)$ is defined by :
\begin{align*}
    \delta_1 : \cfdhat(Y) \rightarrow \mathcal{A} \otimes \cfdhat(Y) 
\end{align*}
defined by :
\begin{gather*}
    \delta_1(\bm{x}) = \sum_{\bm{y} \in \mathfrak{G}(\mathcal{H})} \sum_{\{B\in \pi_2(\bm{x},\bm{y})\;|\; \it{ind}(B,(\rho_{i_1},\cdots,\rho_{i_k}))= 1\}} \#(\mathcal{M}^{B}(\bm{x},\bm{y},(\rho_{i_1},\cdots,\rho_{i_k})) \rho_{i_1} \cdots \rho_{i_k} \bm{y} 
\end{gather*}

Recall that for each $\bm{x} \in \mathfrak{G}(\mathcal{H}), \pi_2(\bm{x},\bm{x})$ i.e. collection of all homotopy classes of $J$-holomorphic \textit{Whitney} disks connecting $\bm{x$} to itself, forms a group where the multiplication is given by concatenation of disks. We can think of $B$ as a domain on the Heegaard surface and a linear combination of the regions in $\Sigma \smallsetminus (\bm{\alpha} \cup \bm{\beta})$. Elements of $\pi_2(\bm{x},\bm{x})$ are called \emph{periodic domains}, which is naturally isomorphic to $H_2(Y,\partial Y)$. A non-trivial class $B\in\pi_2(\bm{x},\bm{y})$ is called \emph{positive} if all its local multiplicities are non-negative. The Heegaard diagram $\mathcal{H}$ is called \emph{provincially admissible} if it has no positive periodic domains with multiplicity $0$ everywhere along $\partial B$. The Heegaard diagram $\mathcal{H}$ is called \emph{admissible} if it has no positive periodic domains.
 
Provincial admissibility of $\mathcal{H}$ ensures that the above mentioned maps $m_i$ and $\delta_1$ are well-defined. Admissibility of $\mathcal{H}$ ensures that $\cfahat(Y,K)$ and $\cfdhat(Y)$ are bounded. Compare with \cite[Chapter 4,5]{osmanifold04} and \cite[Chapter 4] {lot08}. \\

Now the definitions of $\cfdhat(Y_2)$ and $\cfahat(Y_1,K_1)$ in place, we describe the operation between these two modules, mentioned in Equations (1) and (2).
If one of $\cfahat(Y_1,K_1)$ or $\cfdhat(Y_2)$ is bounded, then the \textit{box tensor product} $\cfahat(Y_1,K_1) \boxtimes \cfdhat(Y_2)$ is the $\mathbb{F}$-module $\cfahat(Y_1,K_1) \otimes_{\mathcal{I}} \cfdhat(Y_2)$ equipped with the differential :
\begin{align*}
    \partial^{\boxtimes}(\bm{x} \otimes \bm{y}) = \sum_{k=0}^{\infty} (m_{k+1}\otimes id|_{\cfdhat})(\bm{x}\otimes \delta_k(\bm{y}))
\end{align*}
The finiteness of the sum is ensured by boundedness of any one of $\cfahat(Y_1,K_1),\cfdhat(Y_2)$. 

For the case that we are interested in, writing $\delta_1$ map in terms of the elements of $\mathcal{A}(T^2)$ helps. Let $\rho_{\emptyset} = \iota_0 + \iota_1 = 1$ and then rewrite $\delta_1$ as 
\begin{align*}
    \delta_1 = \sum_i \rho_i \otimes D_i
\end{align*}
where $i$ runs over $\{\emptyset,1,2,3,12,23,123 \} $ and $D_i : \cfdhat \rightarrow \cfdhat$ are called coefficient maps. 

In this notation , the differential in $\cfahat \boxtimes \cfdhat$ can be written in following way:
\begin{align*}
    \partial^{\boxtimes} (\bm{x} \otimes \bm{y}) = \sum_k m_{k+1}(\bm{x}, \rho_{i_1},\rho_{i_2},\cdots,\rho_{i_k}) \otimes D_{i_k} \circ \cdots \circ D_{i_2}\circ D_{i_1}(\bm{y})
\end{align*}
 where $k$ runs over all such sequence $i_1, i_2, \cdots, i_k$ of $k$ elements from $\{ \emptyset, 1,2,3,12,23,123\}$ (including the empty sequence when $k = 0$).

 \begin{subsection}{From $CFK^{-}(S^3,K)$ to $\widehat{CFD}(X_K)$}
 
 We discuss here some materials that stem from the knot Floer homology setup which we need to borrow. Recall that for a knot $K\subset S^3$, $CFK^{\infty}(K)$ is a filtered chain complex, obtained from the Heegaard Floer chain complex of the ambient manifold $S^3$ i.e. $CF^{\infty}(S^3)$ by introducing additional filtration, which is dictated by the knot $K$. Also recall that $CF^{\infty}(S^3)$ is a $\mathbb{F}[U,U^{-1}]$-module over $\mathfrak{G}(\mathcal{H})$ for a Heegaard diagram $\mathcal{H}$ of $S^3$.

We will consider $CFK^{-}(K)$, a chain complex generated by $\mathfrak{G}(\mathcal{H})$ and the differential in this bi-filtered chain complex is given by  
\begin{align*}
    \partial^{-}(\bm{x}) = \sum_{\bm{y} \in \mathbb{T}_\alpha \cap \mathbb{T}_\beta } \sum_{\{\phi \in \pi_2(x,y)|\mu({\phi}) = 1\}} \#(\widehat{\mathcal{M}}({\phi}))\cdot U^{n_w(\phi)} \cdot \bm{y}
\end{align*}
where $\widehat{\mathcal{M}}(\phi)$ denotes the moduli space of $J$-holomorphic disks representing the homotopy type of $\phi$, $\mathcal{M}(\phi)$, quotiented out by the natural action of $\mathbb{R}$ on this moduli space and $\mu(\phi)$ denotes the `expected dimension' of $\mathcal{M}(\phi)$, see \cite{osknot04} for detailed discussions. Setting $U=0$ in the above differential, defines the differential $\hat{\partial}$ for $\cfkhat$, and the homology of that is denoted by $\widehat{HFK}(S^3,K)$.
$CFK^{-}$ denotes $CF^{-}: \{[x,i] \in CF^{\infty}| i<0\}$, with an extra filtration by the Alexander grading $j$.
For a fixed $j$, $\hfkhat(S^3,K,j)$ is the homology of the chain complex  $(\cfkhat(S^3,K,j), \hat{\partial})$, where $\cfkhat(S^3,K,j)$ is generated by all such $\bm{x}$, such that
\begin{align*}
    A(\bm{x}) = \frac{\langle c_1(\underline{s_m}(\bm{x})), [\widehat{F}] \rangle}{2} = j
\end{align*}
where $\widehat{F}$ is a capped-off Seifert surface of $K$ in $S^3_0(K)$. Since $Y$ is a homology sphere, the definition is independent of choice of Seifert surface of $K$ in $Y$. Also, $\underline {s_m}(\bm{x}) \in Spin^c(S^3,K)$, a relative $Spin^c$ structure (see \cite[Section 2.3]{osknot04}).\\

In practice, given a knot $K$ in $S^3$, the knot Floer complex $CFK^{\infty}(S^3,K)$ is viewed in $(i,j)$-plane, where $j$ is the Alexander grading induced by the knot and $i$ denotes the negative power of $U$. The general rule to draw the knot complex is this: on $j$ axis, for a fixed $j$, we put $dim(\hfkhat(S^3,K,j))$-many points at the coordinate $(0,j)$  and then extend that to the whole $(i,j)$-plane by translating through $U$ and $U^{-1}$ (i.e. multiplying $U$ pushes any $[\bm{y},0,m]$ to $[\bm{y},-1,m-1]$, while multiplying by $U^{-1}$ pushes $[\bm{y},0,m]$ to $[\bm{y},1,m+1]$). We assume the complex $CFK^{\infty}(K)$ is \textit{reduced}, which is same as saying $C\{i,j\} = \hfkhat(S^3,K,j-i)$, where $C\{i,j\}$ is all the generators in $CFK^{\infty}(K)$ with $U$-coordinate being $i$, and the $j$-coordinate being $j$. A reduced complex thus means in which there is no length $0$ differential.
This is due to the fact that a filtered chain complex is always filtered chain homotopic to a reduced complex, see \cite[Reduction Lemma]{hw18}. In practice, the boundary maps $\partial^{\infty}$ are drawn by arrows emanating from some generator(s), pointing towards the generator(s) that live(s) in their boundary. A reduced chain complex ensures that the boundary arrows are pointing downwards (i.e. when the boundary map strictly reduces the Alexander grading), pointing to the left (i.e. when the boundary map strictly increases the $U$-power) or both (i.e. when the boundary map points to south-west direction i.e. when the Alexander grading decreases and $U$-power increases, at the same time). 
 
Reduction Lemma implies
\begin{gather*}
  \it{dim}_{\mathbb{F}[U]}(\cfkminus(S^3,K)) = \it{dim}_{\mathbb{F}}(\hfkhat(S^3,K)) = 2n +1, \text{for some $n \ge 0$.}   
\end{gather*}
For a $\mathbb{Z}\oplus\mathbb{Z}$-filtered complex $C$ and a given filtered basis $\{x_i\}$, we call an operation on the basis a \emph{filtered change of the basis} if that operation replaces some basis element $x_j$ by $\sum_{i=1}^{m} a_i y_i$ such that both filtrations of each $a_iy_i$ is less than or equal to those of $x_j$, for all $i =1,2,\cdots, m$, where $\{y_j\}$ is also a filtered basis of $C$. 
We call a $\mathbb{F}[U]$ basis $\{ \xi_0, \xi_1,\xi_2, \cdots, \xi_{2n} \}$ of $\cfkminus$ \emph{vertically simplified} if
\begin{itemize}
    \item $\partial^{\it{vert}}(\xi_{2i -1}) = \xi_{2i}$ (mod $U \cdot C^{-}$) for $i=1,\cdots,n.$ 
    \item $A(\xi_{2i-1}) - A(\xi_{2i}) = k_i > 0.$
    \item $\xi_0$ is the generator of the vertical homology.
\end{itemize}
    where $\xi_i \in C\{i=0\}, i=0,1,\cdots,2n $ and $\partial^{\it{vert}} = \partial^{\infty}|_{C\{i=0\}}$.

Similarly, we can also define a filtered basis  $\{\eta_0, \eta_1,\eta_2, \cdots, \eta_{2n} \}$ of $\cfkminus$ to be  \emph{horizontally simplified} where
\begin{itemize}
    \item $\partial^{\it{hor}} (\eta_{2p-1}) = U^{l_p} \cdot \eta_{2p}$ (mod the associated graded object of $CFK^{-}$ with respect to the Alexander grading $j$, corresponding to $j = A(\eta_{2p-1}) - 1$), for $p = 1,2,\cdots,n.$
    \item $A(\eta_{2p}) - A(\eta_{2p-1}) = l_p > 0.$ 
    \item $\eta_0$ is the generator of the horizontal homology.
\end{itemize}
where $\eta_p \in C\{j=0\} ,p = 0,1,\cdots,2n$ and $\partial^{\it{hor}} = \partial^{\infty}|_{C\{j=0\}}$.

Lipshitz-Ozsv\'ath-Thurston in \cite[Theorem 11.57] {lot08} and Hom in \cite[Lemma 2.1]{hom11} proved that $\cfkminus$ \textit{always} admits vertically and horizontally simplified bases. Also the facts that $C\{i=0\}$ and $C\{j=0\}$ are homotopy equivalent and $\{\ell_1,\ell_2,\cdots,\ell_n\} = \{k_1,k_2,\cdots,k_n\}$ are the same sets, follow from the symmetry of knot Floer homology under reversing the roles of the marked points $w,z$, thus the orientation of the knot, see \cite[Proposition 3.8]{osknot04}.

Now we recall Lipshitz-Ozsv\'ath-Thurston's algorithm from \cite[Theorem 11.26, A.11]{lot08} to find the complete set of generators for $\cfdhat(X_K)$ given $\cfkminus(K)$, where $X_K = S^3 \smallsetminus K$, where the knot is taken to be $r$-framed. 
\begin{theorem}[Theorem 11.26, A.11 in \cite{lot08}]
\label{lotmain}
With notation as above, if $X_K$ denotes the complement of the knot $K$ with an integer framing $r$, then $\cfdhat(X_K)$ has this following description : 
\begin{itemize}
    \item $\iota_0(\widehat{CFD}(X_K))$ is of dimension 2n+1 and is generated by $\{ \xi_0, \xi_1, \cdots ,\xi_{2n}\}$ or \\ $\{ \eta_0,\eta_1, \cdots, \eta_{2n}\}.$
    \item $\iota_1(\widehat{CFD}(X_K))$ is generated by : 
    \begin{align*}
       \bigcup_{i \in \{1,2,\cdots,n\}} \{\kappa_1^i, \kappa_2^i, \cdots, \kappa_{\ell_i}^i\} \cup \bigcup_{j \in \{1,2, \cdots, n\}} \{ \lambda_1^j, \lambda_2^j \cdots,\lambda_{k_j}^j \} \cup \{\mu_1,\cdots,\mu_t\}, where
    \end{align*} 
    \item For each vertical arrow $\xi_{2i-1} \rightarrow \xi_{2i}$ of length $\ell_i$, we have $\kappa_1^i ,\cdots ,\kappa_{\ell_i}^i$ (subspace generated by these is called vertical chain) with following differentials:
    \begin{align*}
        \xi_{2i-1} \xrightarrow{D_1} \kappa_1^i \xleftarrow{D_{23}} \cdots \xleftarrow{D_{23}} \kappa_{\ell_i}^i \xleftarrow{D_{123}} \xi_{2i} 
    \end{align*}
    \item For each horizontal arrow $\eta_{2j-1} \rightarrow \eta_{2j} $ of length $k_j$, we have $\lambda_1^j, \cdots, \lambda_{k_j}^j$ (subspace generated by these is called horizontal chain) with following differentials:
    \begin{align*}
        \eta_{2j-1}\xrightarrow{D_3} \lambda_1^j \xrightarrow{D_{23}} \cdots \xrightarrow{D_{23}} \lambda_{k_j}^j \xrightarrow{D_2} \eta_{2j}
    \end{align*}
    \item{If $t = 2\tau(K) - r$, then we have another additional set of generators $\{ \mu_1,\cdots,\mu_{|t|}\}$ (subspace generated by these is called $unstable$ $chain$) with following differentials: }
    \begin{align*}
        \begin{cases}
        \xi_0 \xrightarrow{D_1} \mu_1 \xleftarrow{D_{23}} \mu_2 \xleftarrow{D_{23}} \cdots \xleftarrow{D_{23}} \mu_t \xleftarrow{D_3} \eta_0 & \quad \text{if $t > 0$} \\
        \xi_0 \xrightarrow{D_{12}} \eta_0 & \quad \text{if $t = 0$} \\
        \xi_0 \xrightarrow{D_{123}} \mu_1 \xrightarrow{D_{23}} \mu_2 \xrightarrow{D_{23}} \cdots \xrightarrow{D_{23}} \mu_{|t|} \xrightarrow{D_2} \eta_0 & \quad \text{if $t < 0$}
        \end{cases}
    \end{align*}
\end{itemize}
where $\tau(K) =$ min$\{p| i_p :C\{i=0,j \leq p\} \rightarrow C\{i=0\}$ induces surjection in homology\}, which is a concordance-invariant defined by Ozsv\'ath-Szab\'o in \cite{os4genus03} and Rasmussen in \cite{rasmussen03}. \\
The gradings are determined as follows: 
\begin{itemize}
    \item The grading set is $G / \lambda^{-1} gr(\rho_{23})^{-r}gr(\rho_{12})^{-1}$ 
    \item Grading of any element $\bm{x}_0$ in $\iota_0(\cfdhat(X_K,r))$, represented by a generator $\bm{x}$ of the knot Floer homology, is determined by Alexander grading $A$ and Maslov grading $M$ of $\bm{x}$ in the knot Floer complex : $gr(\bm{x}_0) = \lambda^{M(\bm{x}) - 2A(\bm{x})} (gr(\rho_{23}))^{-A(\bm{x})}$
\end{itemize}
\end{theorem}

\noindent Next we discuss briefly about the group $G$, of which $\lambda$ is an element, and the grading $gr$ that is associated with the
bordered Floer chain complex $\widehat{CFD}(X_K)$. 
     
 \end{subsection}

 \begin{subsection}{Grading in bordered Floer homology} 
 Here we recall the grading scheme in bordered Floer homology following \cite[Chapter 10]{lot08}.
The grading for elements of a bordered Floer complex, denoted by $gr$, takes values in a non-commutative group $G(\mathcal{Z})$, whose elements are triples of the form $(m;i,j)$ where $m,i,j \in \frac{1}{2}\mathbb{Z}, i+j \in \mathbb{Z} $, where the half integer $m$ is the Maslov component, the pair $(i,j)$ denotes the $spin^c$-component. We will also be interested in $\Tilde{G} = G(\mathcal{Z}) \times \mathbb{Z}$, where the last component reflects the $U$ grading. The group law is defined by : 
\begin{align*}
    (m_1;i_1,j_1;n_1) \cdot (m_2;i_2,j_2;n_2) = (m_1 + m_2 + (i_1j_2 -i_2j_1);i_1+i_2,j_1+j_2; n_1+n_2)
\end{align*}
$G(\mathcal{Z})$ has these grading on Reeb elements : 
\begin{align*}
    gr(\rho_1) &=  (-\frac{1}{2};\frac{1}{2},-\frac{1}{2}) \\
    gr(\rho_2) &= (-\frac{1}{2}; \frac{1}{2}, \frac{1}{2}) \\
    gr(\rho_3) &= (-\frac{1}{2}; -\frac{1}{2}, \frac{1}{2})
\end{align*}
along with this rule that $gr(\rho_I \rho_J) = gr(\rho_I)gr(\rho_J)$ and $gr(\rho_{IJ}) = \lambda gr(\rho_J) gr(\rho_I)$, (where $IJ \in \{12,23,123\}$ ) where $\lambda = (1;0,0) \in G(\mathbb{Z})$.

Recall that the set of all periodic domain is isomorphic to $H_2(Y,\partial Y) \cong \mathbb{Z}$. 

If we call the image of the generator of this group in $\Tilde{G}$ by $g$, Then for a multiplication map $m_{k+1}(x,\rho_{i_1},\cdots,\rho_{i_k}) = y, \Delta_U = m$ in $\cfahat(Y,K)$ ($\Delta_U$ denotes the $i-$filtration shift), we have 
\begin{align*}
    gr(y) = \lambda^{k-1} gr(x)gr(\rho_{i_1})\cdots gr(\rho_{i_k})\cdot u^{-m} \in \langle g \rangle \symbol{92} \Tilde{G}
\end{align*} 
where $u = (0;0,0;-1) \in \Tilde{G}$ (Observe that both $\lambda, u$ are in the centralizer of $\Tilde{G}$). 

If we call the image of the generator of periodic domains in $\Tilde{G}$, $h$, then for $D_{I}$ being a coefficient map from $x$ to $y$ in $\cfdhat(Y,r)$ we have \begin{align*}
    gr(y) = \lambda^{-1} gr(\rho_I)^{-1} gr(x) \in \Tilde{G} / \langle h \rangle
\end{align*}
The box tensor product between $\cfdhat$ and $\cfahat$ of two manifolds with torus boundary is then graded by $ \langle g \rangle / \Tilde{G} \symbol{92} \langle h \rangle$. Every element in this double-coset space is uniquely equivalent to an element of the form $\lambda^a u^b$, for some $a,b \in \mathbb{Z}$ i.e. the grading of that element takes form $(a;0,0;-b)$.

We recall that the \emph{$z$-normalized Maslov grading $N$}, defined by Lipshitz-Ozsv\'ath-Thurston in \cite[Section 11.3]{lot08} can be realized in this way: $N = M- 2A$ (\cite[Equation 11.13]{lot08}) and $N=0$ for the generator of $H_{*}(gCFK^{-}(K)/ U=1) \cong \mathbb{Z}$ (\cite[Equation 11.15]{lot08}), where $M$ denotes the Maslov grading and $A$ denotes the Alexander grading of an element in $gCFK^{-}(K)$, the associated graded object filtered by the Alexander grading $j$. The first coordinate $a$ from the discussion above, is the value of $N$, upto an additive constant. The last coordinate $b$ from above, is the Alexander grading, upto an additive constant. See \cite[11.9]{lot08} for an example showing how one can find the exact Maslov and Alexander grading, using Poincar\'e polynomial, using the fact that weighted Euler characteristics of knot Floer homology is the Poincar\'e polynomial of the knot.

Now, if we have two elements in knot Floer homology of a thin knot, from the splicing formula, whose grading reduce to $(a_1;0,0;b_1)$ and $(a_2;0,0;b_2)$, then $N_1 = a_1 + c_0 = M_1 - 2A_1$, $N_2 = a_2 + c_0 = M_2 - 2A_2$ and $A_1 = b_1 + d_0$, $A_2 = b_2 + d_0$ (where $c_0$ and $d_0$ are some additive constants for $N$ and $A$, respectively).  For thin knots $M_1 - M_2$ should be equal $A_1 - A_2$, which implies that:
\begin{equation}
\label{eq}
    a_1 - a_2 = b_2 - b_1
\end{equation}
 Next we carry out the grading calculations and find two elements in the knot Floer homology using the splicing formula, for which Equation \ref{eq} fails to hold.
     
 \end{subsection}
\section{Proof of Theorem~\ref{thm:main}}
\begin{subsection}{Convention}
Throughout the proof, we assume that for the pattern $T_{p,q} \subset \partial(D^2 \times S^1)$, $p$ and $q$ are relatively prime and that $q >0$.
Additionally, since $K_{-p,q} = -((-K)_{p,q})$, the Floer thinness of $K_{p,q}$ implies the Floer thinness of $K_{-p,q}$, and vice-versa. Here by $-L$ we denote the mirror of $L$. Thus without loss of generality, we can also assume that $p > q > 0 $ and the framing of $K$ is arbitrary.  If $q > p > 0$, let $q = mp+i, m > 0, p > i \geq 1$. By choosing $m$ framing of $K$ as the companion and $T_{p,i}$ as the pattern, we can reduce this case to the previous one.

For technical reasons, we first prove the main statement for the case $q \neq 1$. In the last subsection, we prove the statement for the case when $q=1$.
\end{subsection}

 \begin{subsection}{Doubly Pointed Heegaard diagram of $T_{p,q} \subset S^1 \times S^1$}
  Petkova in \cite{petkova13} and Hom in \cite{hom11} used a bordered Heegaard diagram of $(p,1)$-pattern knot in the solid torus, as in Figure \ref{fig4}, and considered the lifts of the $\bm{\alpha}$ arcs and the $\bm{\beta}$ circle in the universal cover of the genus one Heegaard surface i.e. Euclidean plane, to count the \textit{Whitney} disks between the generators. We follow a similar strategy to find $\cfahat(D^2 \times S^1, T_{p,q})$.

 \begin{figure}
    \centering
    \includegraphics[width=8cm]{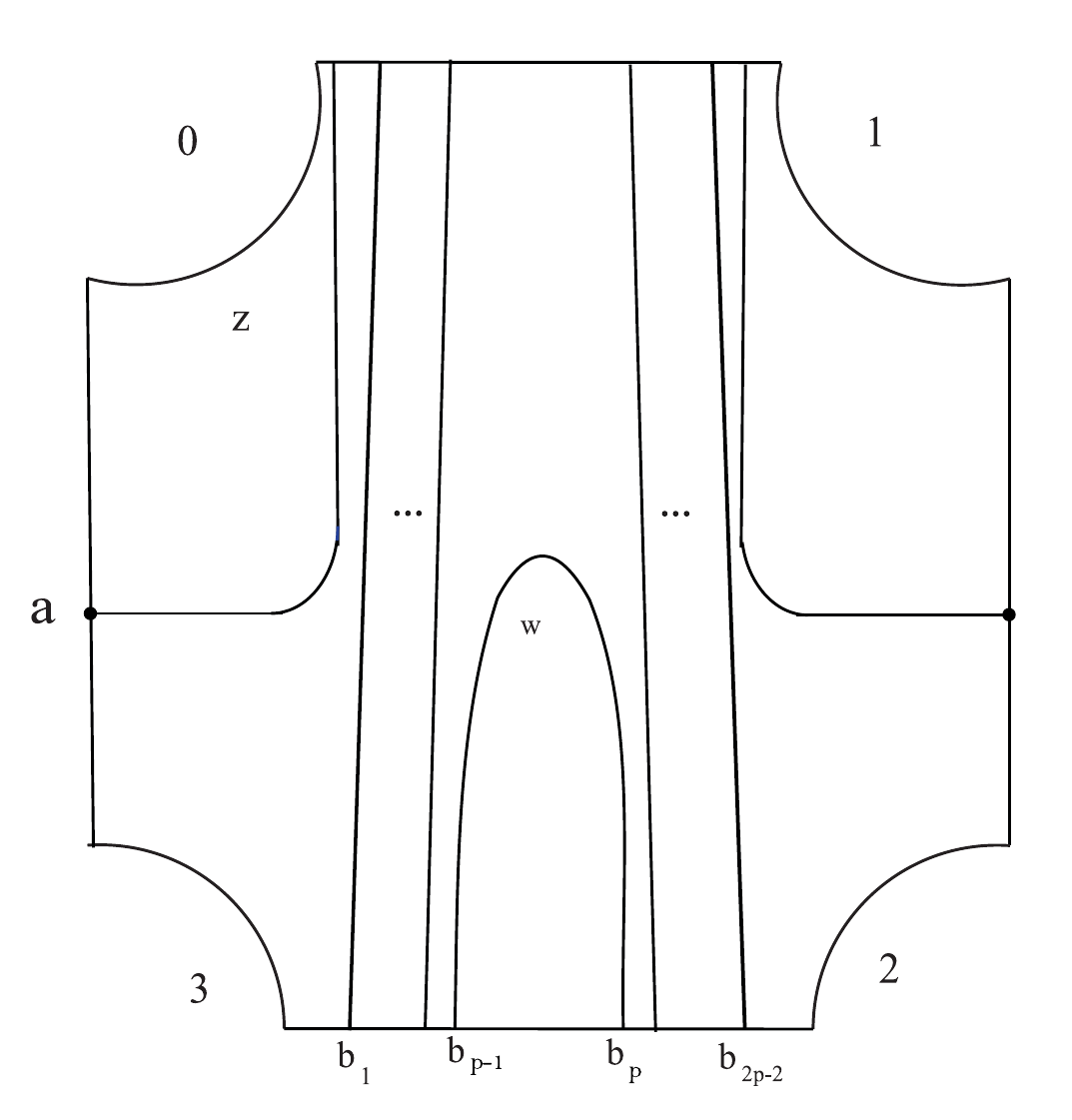}
    \caption{A doubly pointed bordered Heegaard diagram of $(p,1)$ torus knot in $D^2 \times S^1$}
    \label{fig4}
\end{figure} 


 In \cite[Theorem 3.5]{ording06}, Ording describes an algorithm to find a genus one doubly pointed Heegaard diagram for any $(1,1)$ knot. For this paper, we are specifically interested in doubly pointed Heegaard diagrams of torus knots on a torus. To start, one uses a standard form of the knot on the torus. Next, draw the $\beta$ curve on the torus, isotopic to the standard longitude of the torus, ensuring that it avoids the $t_{\beta}$. Here, $t_{\beta}$ represents the part of the knot in the $\beta$-handle body, i.e., the portion obtained by joining $w$ to $z$ without crossing $\beta$. Refer to \cite[Figure 3.11]{ording06} for a step-by-step diagram obtained after applying the algorithm for $T_{5,3}$.


After obtaining the $\beta$ circle for the knot diagram of $T_{p,q}$ following Ording's method, we remove a neighborhood of the vertices of the fundamental domain of the torus. The horizontal and vertical boundary components are then taken as the $\alpha$-arcs, where one initially represented a longitude and the other a meridian of the torus before the removal. Now renaming $w$ in Ording's picture to $z$ and introducing another base point $w$ at the bottom of the picture we obtain a genus one doubly pointed bordered diagram for $T_{p,q} \subset \partial(D^2 \times S^1)$ in the solid torus. It is worth noting that if we disregard the basepoint $w$ and perform an isotopy, we arrive at the standard genus one Heegaard diagram of $S^3$, with $\alpha, \beta$ curves being a longitude and a meridian of $S^3$, respectively. This implies that the bordered diagram we considered is indeed a bordered diagram of $D^2 \times S^1$. The fact that Ording's algorithm produces a doubly pointed Heegaard diagram of $T_{p,q} \subset S^3$ after plugging $w$ basepoint, indicates that we have successfully constructed a doubly pointed bordered diagram of $(D^2 \times S^1, T_{p,q})$ such that $T_{p,q} \subset \partial(D^2 \times S^1)$.

To obtain the standard $(p,q)$ torus knot, join $w$ to $z$ in the complement of the $\beta$ curve, and then connecting from $z$ to $w$ in the complement of the meridian is required. Notably, this operation also involves a reflecting Ording's picture with respect to the the horizontal $\alpha$ arc. For the $(5,3)$ torus knot, this flipped configuration would resemble Figure \ref{fig7} on the torus and Figure \ref{fig8} in the lifted setting. We denote the doubly pointed Heegaard diagram for $T_{p,q} \subset \partial(D^2 \times S^1)$ as $\mathcal{H}(p,q)$.

Below are two pictorial examples of doubly pointed bordered Heegaard diagrams of $\mathcal{H}_{p,q}$ on the fundamental domain of a torus and then in the lift, where $p = 3, q=2$ in the first case (Figure \ref{fig5}, \ref{fig6}) and $ p =5, q = 3$ (Figure \ref{fig7}, \ref{fig8}) in the second case.
\begin{figure}
    \centering
    \includegraphics[width=7.5cm]{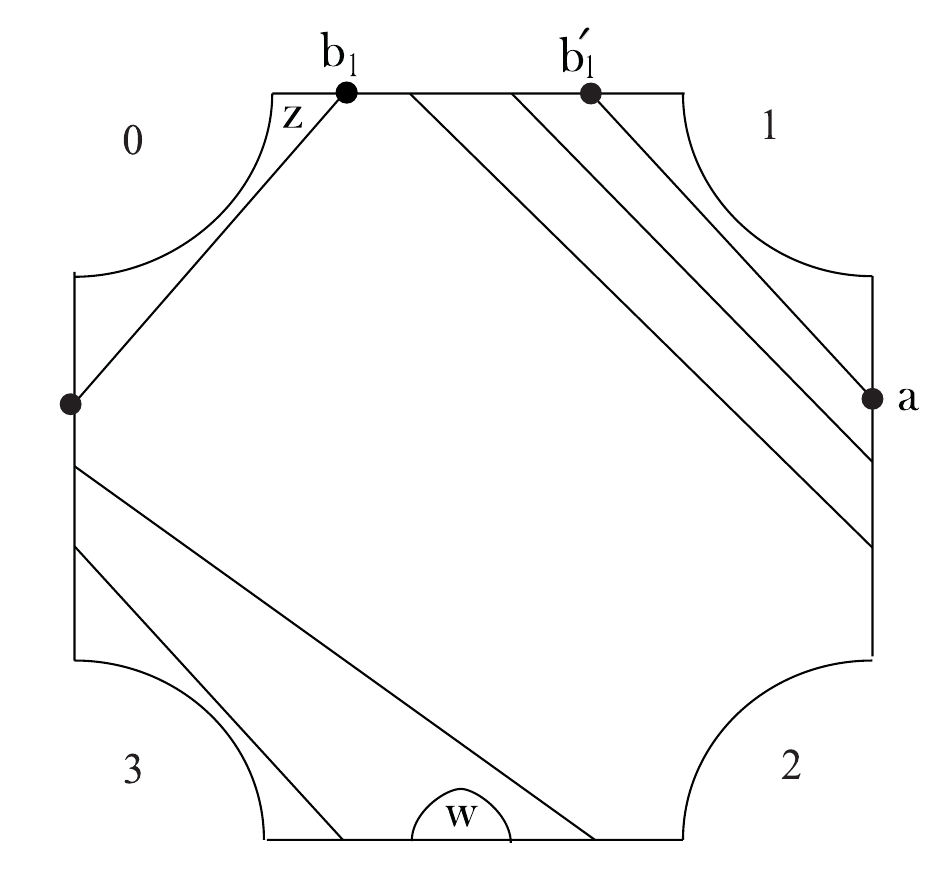}
    \caption{A genus one bordered Heegaard diagram $\mathcal{H}(3,2)$ of $T_{3,2}$}
    \label{fig5}
\end{figure} 

\begin{figure}
    \centering
    \includegraphics[width=12cm]{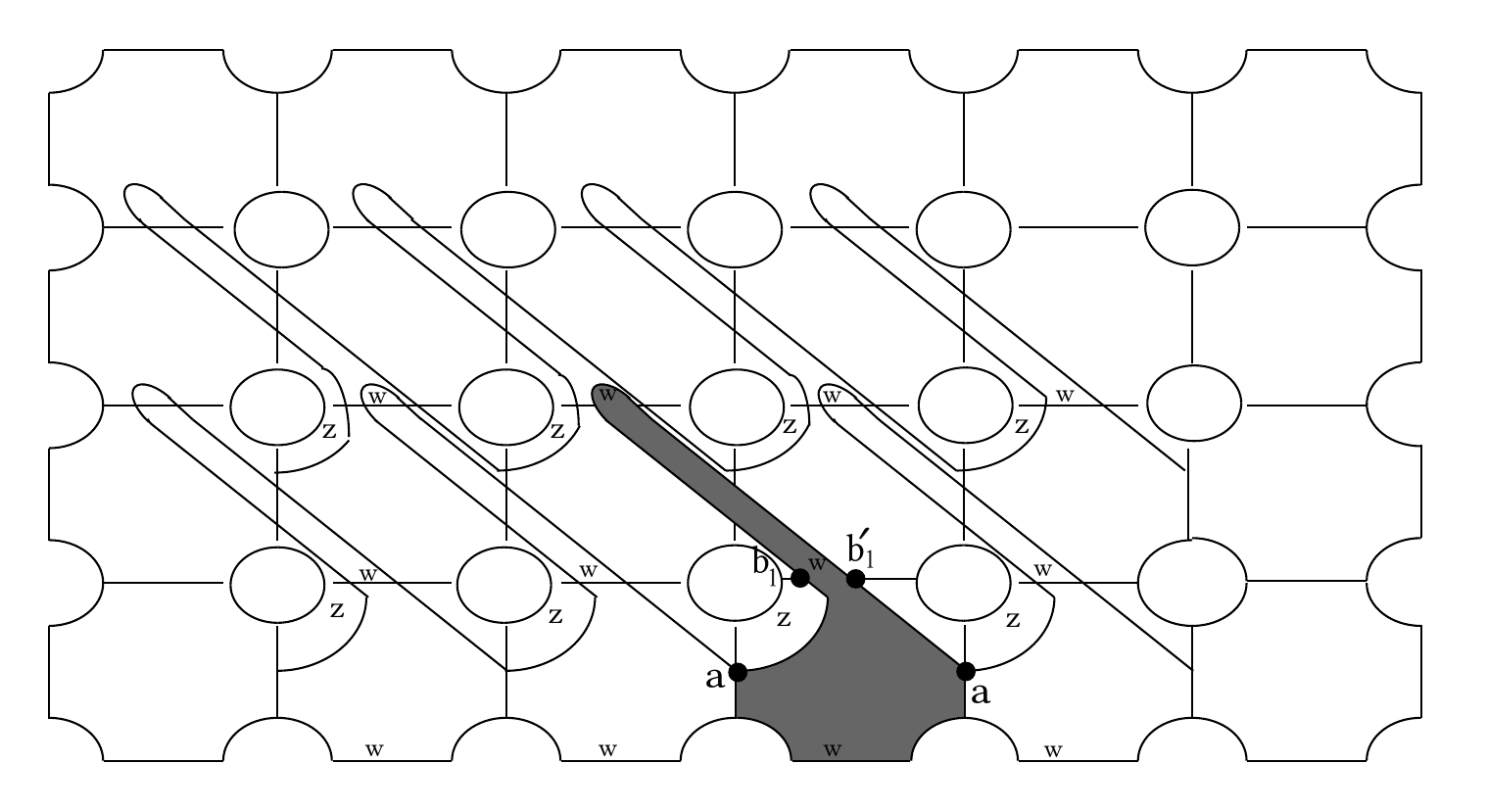}
    \caption{ A part of the lifted bordered Heegaard diagram $\mathcal{H}(3,2)$ of $T_{3,2}$  }
    \label{fig6}
\end{figure}

  Within this context, we designate the sole intersection point between the $\beta$ curve and the $\alpha$ arcs, obtained after isotoping the $\beta$ curve over the $w$-basepoints, as point $a$. It is worth noting that $a$ resides in the $\alpha^a_1$ arc.

If we enumerate the intersection points of the $\beta$ curve with the $\alpha^a_2$ arc starting from the left, we label the first intersection point as $b_1$ and the last one as $b'_1$. Refer to Figures \ref{fig5} and \ref{fig7} for examples when $q \neq 1$ and also Figure \ref{fig4} when $q = 1$.

  \begin{remark}
  \label{remark}
      We will focus on the mentioned intersection points $a, b_1, b'_1$ in $\mathcal{H}(p,q)$ and the local behaviour of $\mathcal{H}(p,q)$ around those. Notably, for $p>1$:  
      \begin{itemize}
      \item The doubly pointed Heegaard diagram $(\Sigma, \alpha, \beta, w, z)$ is \textit{reduced}, meaning there is no Whitney disk connecting two intersection points of $\widehat{CFK}(S^3, T_{p,q})$ without any basepoints. 
          \item The only Whitney disc connecting the intersection point $a$ without any basepoints is the one connecting $a, b'_1$ with $\rho_1$ on its boundary. The same is true for $b'_1$.
          \item There are no Whitney disc connecting $b_1$ without basepoint, which has $\rho_2$ on its boundary.
      \end{itemize}

\noindent Note that the first remark follows since we can pull the $\alpha, \beta$ curves tight on the Heegaard torus (or the lifted curves on the universal cover) i.e. isotope the attaching curves, to avoid empty bigons. Also refer to \cite[Reduction Lemma]{hw18} for an algebraic explanation in the general case. Both the second and the third remark follows from the nature of $\beta$ curve as Ording's algorithm and the descriptions of the mentioned intersection points. These observations contribute to the understanding of the local characteristics of $\mathcal{H}(p,q)$ in the specified conditions.
  \end{remark}

To find the element that survives in the homology, Petkova in \cite{petkova13} and Hom in \cite{hom11} find all the $\mathcal{A}_{\infty}$-relations from a bordered diagram of $(p,1)$ pattern in $D^2 \times S^1$. For our case, we will only be interested in three mentioned generators in $\widehat{CFD}(D^2 \times S^1, T_{p,q})$ and $\mathcal{A}_{\infty}$-relations concerning them. For any $(p,q)$, it is not always easy to find all the disks and thus all the $\mathcal{A}_{\infty}$-relations. Instead we will be looking into a few specific relations with one eye to our goal, coming from the bordered Heegaard diagram $\mathcal{H}(p,q)$.

\begin{lemma} \label{lemma9}
In $\cfahat(D^2\times S^1, T_{p,q})$, we have the following $\mathcal{A}_{\infty}$-relations :
\begin{align*}
    m_3(a, \rho_3,\rho_2) = U^{n_w}\cdot a  \hspace{.3in} &i.e. \Delta_U = n_w \\ 
    m_4(a, \rho_3,\rho_2,\rho_1) = U \cdot b_1 \hspace{.3in} & i.e. \Delta_U = 1\\
    m_2(a, \rho_1) = b_1' \hspace{.3in} & i.e. \Delta_U = 0\\
    \end{align*}
    where $n_w$ is the number of $w$ basepoints in the primitive positive periodic domain of $\mathcal{H}(p,q)$, and by primitive we mean the generator of $\pi_2(a,a) \cong \mathbb{Z}.$
\end{lemma}
As mentioned before, $\Delta_U$ denotes the $i-$filtration shift, which is the number of times the associated domain cross the $w-$basepoint. Also note that the periodic domain mentioned here is a domain that does not contain any $z$ basepoints.

\begin{proof}
  For a given genus one doubly pointed bordered Heegaard diagram, we can look at the fundamental domain of the torus and find a periodic domain joining $a$, bounded by $\alpha$-arcs and the $\beta$ curve. Notice that since $a \in \iota_0(\widehat{CFA}(D^2 \times S^1, T_{p,q}))$, we get that the primitive domain connecting $a$ contains $\rho_2, \rho_3$ on its boundary. Thus we get the first $\mathcal{A}_{\infty}$-relation (See Figure \ref{fig6}, \ref{fig8} for reference).
  
Similarly, from the description of $b_1$, one can see that a domain exists which is bounded by $\alpha$-arcs to the right and $\beta$ curves to the left, joining $a$ and $b_1$, picking up $\rho_3,\rho_2,\rho_1$ in the process. Also, the $U$ power takes care of number of $w$ encountered inside the domain, which is 1. Combining these, we get the second $\mathcal{A}_{\infty}$-relation. 

The third relation follows from the description of $b_1'$ as well, where Whitney disks can be found bounding $a$ and $b_1'$, which picks up $\rho_1$ in the boundary.
   \end{proof}
   
Note that $b_1'$ in Figure \ref{fig4} is the intersection point $b_{2p-2}$.

\begin{figure} [h!]
    \centering
    \includegraphics[width=7cm]{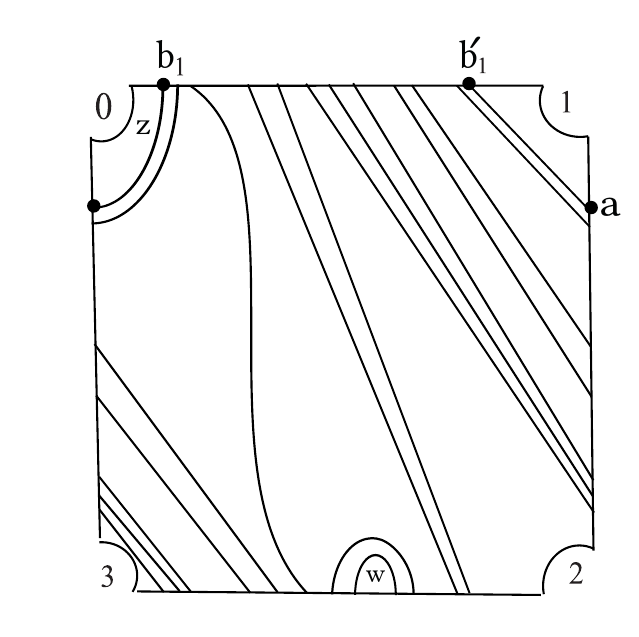}
    \caption {A genus one bordered Heegaard diagram $\mathcal{H}(5,3)$ of $T_{5,3}$}
    \label{fig7}
\end{figure}

\begin{figure}[h!]
    \centering
    \includegraphics[width=12cm]{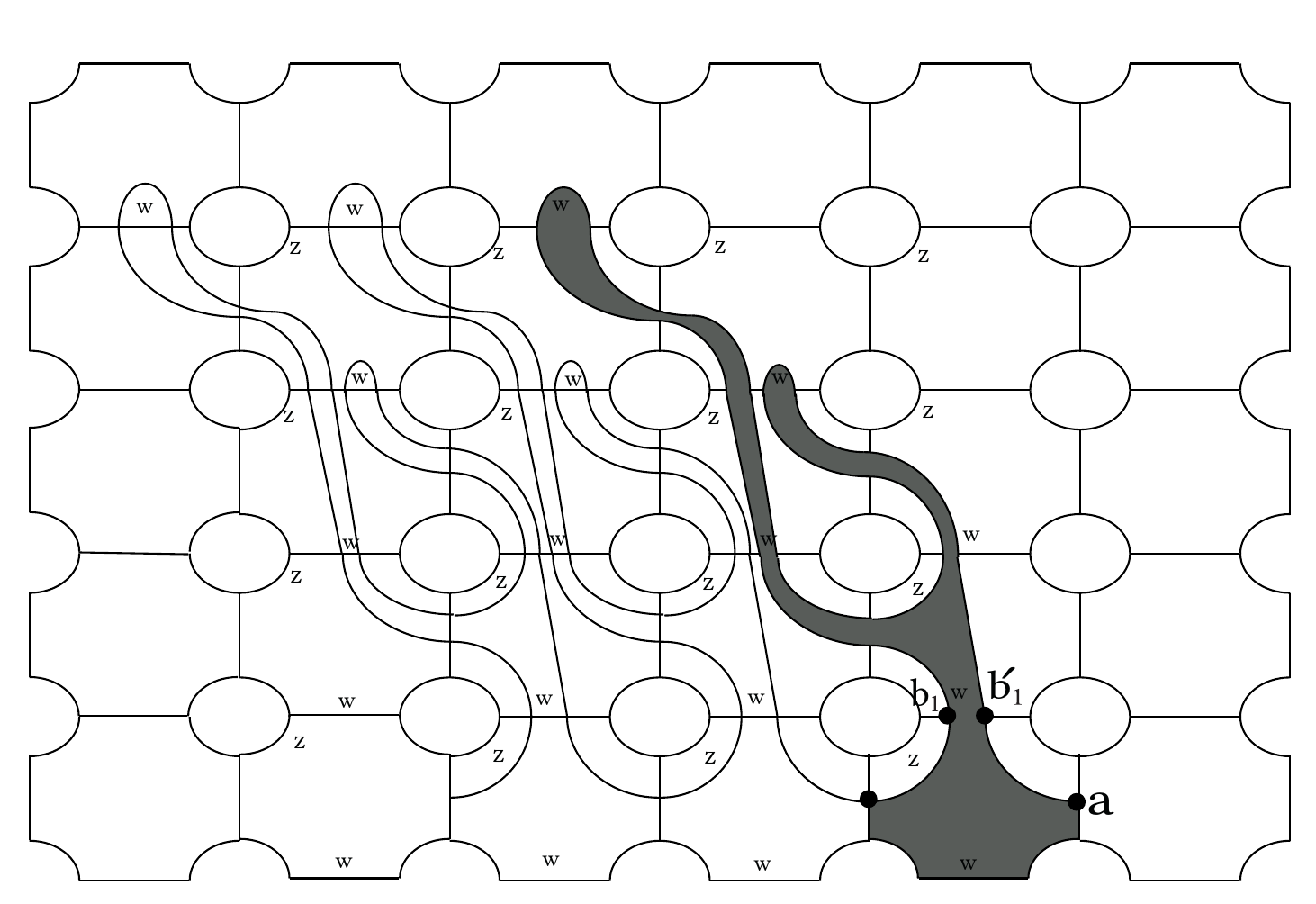}
    \caption{A part of the lifted $\mathcal{H}(5,3)$ of $T_{5,3}$}
    \label{fig8}
\end{figure}

\begin{lemma}
\label{lemma10}

 Multiplicity of $w$ in the primitive periodic domain of $\mathcal{H}(p,q)$ is $vx + 1$, where $x,y,u,v$ are unique positive integers such that $p = x+y, q = u+v $ such that $vx-uy = 1$. 
\end{lemma}
\begin{proof}
We consider the doubly pointed Heegaard diagram $(\Sigma, \alpha,\beta, z, w)$ for a torus knot $T_{p,q}$ using Ording's cat's cradle description from \cite{ording06}. As mentioned earlier, using the algorithm one can obtain doubly pointed bordered Heegaard diagram $\mathcal{H}(p,q)$ of $T_{p,q} \subset D^2 \times S^1$. We consider the lifts $\Tilde{\alpha}, \Tilde{\beta}$ of $\alpha, \beta$ circles, respectively, in the universal covering of a torus i.e. in $\mathbb{R}^2$. Note that without loss of generality we can assume that there is no bigon connecting two intersection points between $\Tilde{\alpha}$ and $\Tilde{\beta}$ with no $z$ and $w$ base points, since we have argued that the chain complex can be assumed to be reduced i.e. devoid of such bigons.

Note that torus knots are $L$-space knots since positive surgery along torus knots with certain coefficient produces lens space, by \cite{moser71}. Ozsv\' ath-Szab\'o (\cite{oslens05}) showed that the knot Floer complex of $L$-space knots forms a `staircase' complex and hence $dim_{\mathbb{F}}(\hfkhat(T_{p,q}, j))$ is at most 1, for all $j$. Since the weighted Euler characteristics of knot Floer homology is the Alexander polynomial of the knot, the above discussion implies that $dim_{\mathbb{F}}(\hfkhat(T_{p,q}))$ is equal the number of non-zero terms in $\Delta_{T_{p,q}}$, which is $2vx - 1$ by \cite[Corollary 2.6]{song17}, where $x,y,u,v$ are unique positive integers such that $p = x+y, q = u+v $ such that $vx-uy = 1$.

Observe that in the lifted picture, whenever $\Tilde{\beta}$ crosses a $\Tilde{\alpha}$, we get two generators in $\hfkhat(T_{p,q})$ and the bigon connecting them has a $w$ in it. There are $2vx - 1$ intersection points between $\Tilde{\beta}$ and $\Tilde{\alpha}$. Barring the intersection $a$ (which is the generator of $\widehat{HF}(S^3)$), for every two such, there is a lift of the basepoint $w$ which is contained in the bigon representing the primitive periodic domain. Hence, the number of $w$'s in that domain is equal to $ \frac{2vx-1-1}{2} + 2 = vx+1$, where other then the periodic domain, one $w$ multiplicity comes from a $w$ basepoint lying inside the fundamental domain of torus, bounded by the boundary of the periodic domain. The other $w$ basepoint stays inside this periodic domain according to the algorithm of Ording in \cite{ording06}. See Figure \ref{fig9} where the initial and the end $\beta$ strands of the boundary of this periodic domain are drawn. 
\end{proof}

 

  

 \begin{figure}[h!]
    \centering
    \includegraphics[width=13cm]{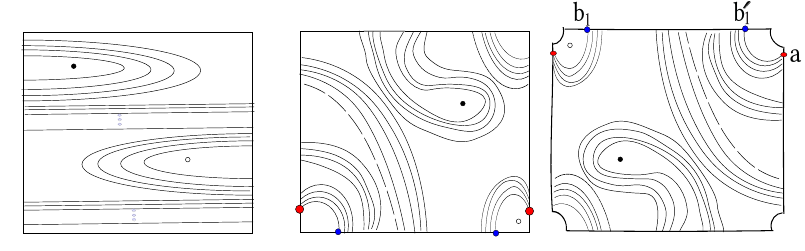}
    \caption {The left most 
    figure shows a \emph{normal} form of $T_{p,q}$ on the fundamental domain where the $i$-th strand on the left vertical boundary, counting from top, gets identified with $p+i$ mod $(q-1)$-th strand on the right vertical boundary.  Second 
    picture shows a schematic \emph{normal} form of a doubly pointed knot diagram for $T_{p,q}$ on the fundamental domain. The right most picture shows the eventual doubly pointed `bordered' fundamental domain, on which one can apply Ording's algorithm to find $\mathcal{H}(p,q)$ i.e. the doubly pointed bordered diagram of $T_{p.q} \subset D^2 \times S^1$. The shaded dot represents $w$ basepoint, the un-shaded dot represents $z$ basepoint. Note that the middle and the right most pictures only shows the relevant portions of the domain }
    \label{fig9}
\end{figure}

 \end{subsection}


Observe that to make sure that an element, say $\gamma$ in $\cfkhat$, is a non-zero element in knot Floer homology, it is enough to check that all the elements in $\partial(\gamma)$ contains a \emph{non-zero} $U$ power i.e. $\partial^{vert}(\gamma)=0$ \emph{and} there is no element $\alpha \in \cfkhat$ such that $\partial^{vert}(\alpha)=\gamma$.

Let $K$ be our companion knot in $S^3$. For the case $q \neq 1$, we are going to find the desired generators, separating our search in three cases : when $\varepsilon(K) = 1$, when $\varepsilon(K) =0$ and when $\varepsilon(K)=-1$. 

 We recall the definition of $\varepsilon$, a knot concordance invariant, defined by Hom in \cite{hom14}. To define $\varepsilon$, it is necessary to recall another knot concordance invariant $\nu$, defined by Ozsv\'ath-Szab\'o.  Recall that 
 \[\nu(K)=min \{ s| p_s : C\{ max \{i, j-s\} = 0\}  \rightarrow C\{i=0\} \text{ induces a surjection in homology}\} \] where $p_s$ is the projection map onto the $i$ coordinate. Notice that $\nu(K) = \tau(K)$ or $\tau(K)+1$, see \cite{os4genus03}. Then
 \begin{align*}
     \varepsilon(K) = 
 \begin{cases}
 -1 & \quad \text{if $\nu(K) = \tau(K) + 1 $} \\
 0 & \quad \text{if $\nu(K) = \tau(K)$ and $\nu(-K) = \tau(-K)$ } \\
 1 & \quad \text{if $\nu(-K) = \tau(-K) + 1 $}
     \end{cases}
 \end{align*} 

\vspace{.1in}

\noindent Note that for the proof from now on, we will make use the notations for various chain elements introduced in Theorem \ref{lotmain}.

\begin{subsection}{\underline{\textbf{Case 1.1: $\varepsilon(K)=1, t = 2 \tau(K) - r \neq 0$}}} 
\label{case1}

For $\varepsilon(K) = 1$, we use \cite[Lemma 3.2]{hom11} to find $\{ \xi_i \}$, a \emph{vertically simplified} $\mathbb{F}[U]$ basis of $\cfkminus$, with the following properties, after possible renaming, 
 \begin{itemize}
    \item $U^{k}\cdot \xi_2$ is the generator of the homology of $C\{j=0\}$, for some $k$. 
    \item there exists $\xi_1$ such that $\partial^{vert} \xi_1 = \xi_2.$ 
    \item $\xi_0$ is the generator of the homology of $C\{i=0\}.$
 \end{itemize}

\begin{lemma}
\label{lemma} 
When $\varepsilon(K) = 1$ and $t \neq 0$, $a \otimes \xi_2$ defines a non-zero generator in the homology.  
\end{lemma}
\begin{proof}
  First we look at the description of the generators of $\cfdhat(X_K,r)$ and we observe the corresponding differentials with coefficients to and from  $\xi_2$ in it. The immediate incoming and outgoing arrows to and from $\xi_2$ consist of coefficients $D_{123},D_{23}, D_3,D_2$. Note that since $\xi_2$ is the generator of the horizontal homology, we would only have to consider the vertical and the unstable chain of $\cfdhat(X_K,r)$.
  
To start with, there is no $\mathcal{A}_{\infty}$-relation such as $m_k(a,\rho_{123},\cdots) = c$, or $m_k(d,\cdots,\rho_{123}) = a$ for some $c,d \in CFD^{-}(D^2\times S^1, T_{p,q})$. The former one is since for any Whitney disk without the base points connecting $a$, $\Tilde{\beta}$ always has negative slope in $\mathcal{H}(p,q)$. The latter one since end of $\rho_{123}$ and $a$ lie in distinct idempotents. As per the algorithm, $\beta$ always runs along the standard form of the knot and only changes its direction around $w$. One can see that no such Whitney disks can exist which has the relation $\rho_{123}$ precedes or succeeds $a$ (in fact only $\rho_1$, $\rho_2$ or $\rho_3$ precedes or succeeds $a$ in a relation).


Also, a $\mathcal{A}_{\infty}$-relation $m_{k+2}(a,\rho_3,\underbrace{\rho_{23},\rho_{23},\cdots,\rho_{23}}_{k}) = c$ , for some $c$ is not possible, since a Whitney disk in $\mathcal{H}(p,q)$, starting from $a$ with $\alpha$'s to the right and $\beta$ to the left and having $\rho_{23}$ in it, should contain $\rho_2$ or $\rho_1$ as well.

In a similar way, an relation of the form $m_{k+2}(c, \underbrace{\rho_{23}, \cdots, \rho_{23}}_{k}, \rho_2) = a$ is not possible as a Whintey disk connecting $a$ with only some number of $\rho_{23}$ and $\rho_2$ on its boundary would contain non-zero number of $w$ basepoints.


Combining these observations completes the proof. 
\end{proof}


Let $\xi_{2i}$ be some element in the \emph{vertically simplified} basis of $\cfkminus(K)$ such that
\[
        \xi_{2i-1} \xrightarrow{D_1} \kappa_1^i \xleftarrow{D_{23}} \cdots \xleftarrow{D_{23}} \kappa_{\ell_i}^i \xleftarrow{D_{123}} \xi_{2i} 
    \]

\noindent and let $A(\xi_{2i}) = d, M(\xi_{2i}) = m$ (by this we mean that A is the Alexander grading and M is the Maslov grading of the element of the knot Floer homology of $K$, that represents the element $\xi_{2i}$). Then, using \cite[Theorem A.11]{lot08}, we get that $gr(\xi_{2i}) = \lambda^{m-2d} \cdot gr(\rho_{23})^{-d} \in G(\mathcal{Z})/  \lambda^{-1} gr(\rho_{23})^{-r}gr(\rho_{12})^{-1}$.

 Next, we look at the \emph{unstable chain} of $\iota_1$ part of $\cfdhat(X_K,r)$, mentioned in Theorem \ref{lotmain}. If $t < 0$, then $ \mu_1 = D_{123} \cdot \xi_0$. If $ t > 0$, then $\mu_1 = D_1 \cdot \xi_0$. We calculate the grading of $\mu_1$ for both cases.

Recall that Lemmas \ref{lemma9} and \ref{lemma10} implies that \[ m_3(a, \rho_3,\rho_2) = U^{vx+1} \cdot a \]

Which implies that the image of the group of periodic domain in $\tilde{G}$ is generated by 
\[u^{-(vx+1)} \cdot \lambda \cdot gr(\rho_3) \cdot gr(\rho_2) = u^{-(vx+1)}\cdot gr(\rho_{23})\]

Thus the grading set for $\cfahat$ is isomorphic to  $u^{-(vx+1)} \cdot gr(\rho_{23}) \symbol{92} \Tilde{G}$.
Also as mentioned above, the grading set for $\cfdhat(X_K,r)$ is  $\tilde{G} / \lambda^{-1}\cdot gr(\rho_{23})^{-r} \cdot gr(\rho_{12})^{-1}$, when the framing of the companion knot is $r$.

Then Lemma \ref{lemma9} implies
\begin{align*}
    gr(b_1) = \lambda \cdot u^{-1} \cdot gr(\rho_{23}) \cdot gr (\rho_1) \\
    \sim \lambda \cdot u^{-1} \cdot u^{vx+1} \cdot gr(\rho_1) \\
    = \lambda \cdot u^{vx} \cdot gr(\rho_1)
\end{align*} 

One can recover the equivalency by multiplying $u^{(vx+1)} \cdot gr(\rho_{23})^{-1}$ to $gr(b_1)$ from left, and the fact that $u$ lives in the commutator subgroup of $\tilde{G}$.

From the description of $\xi_2$ and using Theorem \ref{lotmain}, we get that \begin{align*}
    gr(\xi_2) = \lambda^{-2\tau_K + 2\tau_K} \cdot gr(\rho_{23})^{\tau_K} = gr(\rho_{23})^{\tau_K}
\end{align*} 
since the generator of homology of $C\{j=0\}$ has Alexander grading $-\tau_K$ and Maslov grading $-2\tau_K$ and $\xi_2$ is a representative of that.
Thus, since $gr(a) = (0;0,0;0)$, 
\begin{equation}
\label{eq3}
    gr(a \otimes \xi_2) = gr(\rho_{23})^{\tau_K} \sim u^{\tau_K(vx+1)} = (0;0,0;-\tau_K(vx+1))
\end{equation}

We find another non-zero generator in the homology and its grading to compare with $gr(a\otimes \xi_2)$.

We recall from previous discussion that $\varepsilon(K) = 1$ implies that 
\begin{itemize}
    \item $U^{k}\cdot \xi_2$ is the generator of the homology of $C\{j=0\}$, for some $k$ i.e. $U^k\cdot \xi_2 = \eta_0$. 
    \item there exists $\xi_1$ such that $\partial^{vert} \xi_1 = \xi_2.$ 
    \item $\xi_0$ is the generator of the homology of $C\{i=0\}$.
 \end{itemize}


 Also note since the homology of $CFK^{\infty}\{j=0\}$ is isomorphic to the homology of $CFK^{\infty}\{i=0\}$, which is $\hfkhat(K)$. Hence the second item from the implication of $\varepsilon(K) = 1$ implies for some $m$, $U^m \xi_0:= \eta_{2j} \in Im(\partial^{hor})$, for some $j$. 

 i.e. \[\cdots \xrightarrow{D_{23}/ D_3} \lambda^j_{k_j} \xrightarrow{D_2} \eta_{2j} \]

\begin{lemma} \label{epsilon1}
For $\varepsilon(K) = 1$, the following generators are non-trivial in $\hfkhat(S^3, K_{p,q})$:
\begin{itemize}
    \item $b_1 \otimes \mu_1$, for $t < 0, t>1$,
    \item $b_1' \otimes \lambda^j_{k_j}$, for any value of $t$.
    where $\lambda^j_{k_j}$ is as described above.
\end{itemize}

\end{lemma}

\begin{proof}

We consider the cases when $t<0, t>1$ together. Notice that there is no possible algebra relation of the form $m_2(b_1, \rho_2) = c$, $m_{k+1}(b_1, \underbrace{\rho_{23},\cdots}_k) = c$, $m_{k+2}(b_1, \underbrace{\rho_{23},\cdots, \rho_2}_k) = c$, $m_2(c, \rho_{123})=b_1$ or $m_{k+2}(c, \rho_3, \underbrace{\rho_{23}, \cdots}_{k>0}) = b_1$ for some $c$. This is because any Whitney disk, containing no $z, w-$ basepoints, connecting $b_1$ has a $\rho_3$ on its boundary. But if the boundary also contain $\rho_{23}$, then the disk would contain $w-$basepoints.

For the remaining case, we note that there are no algebra relations of the form $m_2(b'_1, \rho_2)= c$, $m_{j+1}(c, \underbrace{\rho_{23}, \cdots}_j) = b'_1$, or $m_{j+2}(c, \rho_3, \underbrace{\rho_{23}, \cdots}_j )=b'_1$, for some $c$. This is because any Whitney disk, containing no $z, w-$ basepoints, connecting $b'_1$ has a $\rho_1$ on its boundary. 
\end{proof}




Now we proceed with the grading calculation of the above generators. By Theorem \ref{lotmain}, $gr(\xi_0)$ equals to $\lambda^{-2\tau_K} \cdot gr(\rho_{23})^{-\tau_K}$.

For the case  $t < 0$,
\begin{align*}
    \mu_1 = D_{123}\cdot \xi_0 \Rightarrow gr(\mu_1) &= \lambda^{-1} \cdot gr(\rho_{123})^{-1}\cdot gr(\xi_0) \\
   &= \lambda^{-1} \cdot gr(\rho_{123})^{-1} \cdot \lambda^{-2\tau_K} \cdot gr(\rho_{23})^{-\tau_K} \\
   &= \lambda^{-1} \cdot \lambda^{-1} \cdot gr(\rho_1)^{-1} \cdot gr(\rho_{23})^{-1} \cdot \lambda^{-2\tau_K} \cdot gr(\rho_{23})^{-\tau_K} \\
   & = \lambda^{-2-2\tau_K} \cdot gr(\rho_1)^{-1} \cdot gr(\rho_{23})^{-\tau_K-1}
   \end{align*}
   
Then
\begin{align*}
    gr(b_1 \otimes \mu_1) &= gr(b_1) \cdot gr(\mu_1) \\
    &= \lambda \cdot u^{vx} \cdot gr(\rho_{1}) \cdot  \lambda^{-2-2\tau_K} \cdot gr(\rho_1)^{-1} \cdot gr(\rho_{23})^{-\tau_K-1}\\
    & = \lambda^{-1-2\tau_K} \cdot u^{vx} \cdot  gr(\rho_{23})^{-\tau_K-1} \\
    & \sim u^{-(vx+1)(\tau_K+1)} \cdot gr(\rho_{23})^{\tau_K+1} \cdot \lambda^{-1-2\tau_K} \cdot u^{vx} \cdot  gr(\rho_{23})^{-\tau_K-1}\\
    & = \lambda^{-1-2\tau_K} \cdot u^{-\tau_K(vx+1) - 1}\\
    &= (-1-2\tau_K;0,0;\tau_K(vx+1)+1)
\end{align*}

Comparing $ gr(a \otimes \xi_2)$ from Equation \ref{eq3} and $ gr(b_1 \otimes \mu_1)$, we get that for Equation \ref{eq} to hold, the following equality needs to hold 

\[1+2\tau_K = 2\tau_K(vx+1)+1\]

 \[\Rightarrow \tau_K = 0 \]
Which is because $vx \neq 0$, as both $v,x$ are positive integers.

We will deal with the case when $\tau_K = 0$ when we deal with $K$ such that $t \neq 0$ and $\varepsilon(K) = 0$ (see Lemma \ref{lemma13}, and the calculations right after), as $\varepsilon(K) = 0 \Rightarrow \tau_K = 0$, see Remark \ref{tau0}.

For the case when $t > 1$,

\begin{align*}
    \mu_1 = D_1 \cdot \xi_0 \Rightarrow gr(\mu_1) &= \lambda^{-1} \cdot gr(\rho_1)^{-1} \cdot gr(\xi_0) \\
    &= \lambda^{-1} \cdot gr(\rho_1)^{-1} \cdot \lambda^{-2\tau_K} \cdot gr(\rho_{23})^{-\tau_K} \\
    &= \lambda^{-2\tau_K - 1} \cdot gr(\rho_1)^{-1} \cdot gr(\rho_{23})^{-\tau_K}
    \end{align*}
Then
\begin{align*}
    gr(b_1 \otimes \mu_1) &= gr(b_1) \cdot gr(\mu_1) \\
    &= \lambda \cdot u^{vx} \cdot gr(\rho_1) \cdot  \lambda^{-2\tau_K - 1} \cdot gr(\rho_1)^{-1} \cdot gr(\rho_{23})^{-\tau_K}\\
    &= \lambda^{-2\tau_K} \cdot u^{vx} \cdot gr(\rho_{23})^{-\tau_K} \\
    & \sim u^{-\tau_K(vx+1)} \cdot gr(\rho_{23})^{\tau_K} \cdot \lambda^{-2\tau_K} \cdot u^{vx} \cdot gr(\rho_{23})^{-\tau_K}  \\
    &= \lambda^{-2\tau_K} \cdot u^{-\tau_Kvx - \tau_K + vx}\\
    &= (-2\tau_K; 0,0; \tau_Kvx + \tau_K - vx)
\end{align*}

Comparing $gr(a\otimes \xi_2)$ from Equation \ref{eq3} and $gr(b_1\otimes \mu_1)$, we get that Equation \ref{eq} happens if only if \[vx(2\tau_K - 1) =  0\] which is not true since $v,x$ are positive integers and $2\tau_K - 1$ is an odd integer.

Note that since $b_1\otimes \mu_1$ may be trivial in the homology if $t = 1$, we separately deal with that case by considering the generator $b_1' \otimes \lambda^j_{k_j}$.

 \[\cdots \xrightarrow{D_{23}/ D_3} \lambda^j_{k_j} \xrightarrow{D_2} \eta_{2j}= U^m \cdot \xi_0\] \[\Rightarrow gr(\lambda^j_{k_j}) = \lambda \cdot gr(\rho_2)\cdot gr(\eta_{2j}) = \lambda \cdot gr(\rho_2)\cdot gr(\xi_0) = \lambda \cdot gr(\rho_2) \cdot \lambda^{-2\tau_K}\cdot gr(\rho_{23})^{-\tau_K}\]


\begin{align*}
    gr(b'_1 \otimes \lambda^j_{k_j})  = gr(b'_1) \cdot gr(\lambda^j_{k_j}) &= gr(\rho_1) \cdot \lambda \cdot gr(\rho_2) \cdot \lambda^{-2\tau_K}\cdot gr(\rho_{23})^{-\tau_K}\\
    &= \lambda^{1-2\tau_K} \cdot gr(\rho_{12})\cdot gr(\rho_{23})^{-\tau_K}\\
    & \sim \lambda \cdot gr(\rho_{12}) \cdot gr(\rho_{23})^{-\tau_K} \cdot \lambda^{-1} \cdot gr(\rho_{23})^{-r} \cdot gr(\rho_{12})^{-1} \\
    &=\lambda^{-2\tau_K} \cdot (gr(\rho_{12})\cdot gr(\rho_{23})^{-1} \cdot gr(\rho_{12})^{-1})^{(r+\tau_K)}
\end{align*}

Note that above we used that $gr(b_1') = gr(\rho_1)$, by Lemma \ref{lemma9}.

\begin{align*}
    gr(\rho_{12})\cdot gr(\rho_{23})^{-1} \cdot gr(\rho_{12})^{-1} &= gr(\rho_1) \cdot gr(\rho_2)\cdot \lambda^{-1} \cdot gr(\rho_2)^{-1}\cdot gr(\rho_3)^{-1}\cdot gr(\rho_{12})^{-1}\\
    &= \lambda^{-1} \cdot gr(\rho_1) \cdot gr(\rho_{123})^{-1}\\
    &= \lambda^{-1} \cdot gr(\rho_1)\cdot \lambda^{-1} \cdot gr(\rho_1)^{-1}\cdot gr(\rho_{23})^{-1}\\
    &=\lambda^{-2} \cdot gr(\rho_{23})^{-1}\\
    & \sim \lambda^{-2} \cdot gr(\rho_{23})^{-1} \cdot u^{-(vx+1)} \cdot gr(\rho_{23})\\
    &= \lambda^{-2} \cdot u^{-(vx+1)}
\end{align*}
\begin{align}
\label{eq4}
  \Rightarrow gr(\rho_{12})\cdot gr(\rho_{23})^{-1} \cdot gr(\rho_{12})^{-1} =   \lambda^{-2} \cdot u^{-(vx+1)}
\end{align}
Here $r$ is the framing of $K$. Since $t = 2\tau_K - r =1 \Rightarrow r = 2\tau_K-1$. 
\begin{align*}
   gr(b'_1 \otimes \lambda^j_{k_j})&=\lambda^{-2\tau_K}\cdot(gr(\rho_{12})\cdot gr(\rho_{23})^{-1} \cdot gr(\rho_{12})^{-1})^{(r+\tau_K)}\\
   &= \lambda^{-2\tau_K} \cdot (\lambda^{-2} \cdot gr(\rho_{23})^{-1})^{(r+\tau_K)}\\
   &= \lambda^{-4\tau_K-2r}\cdot gr(\rho_{23})^{-(\tau_K+r)}\\
   &\sim u^{-(vx+1)(\tau_K+r)}\cdot gr(\rho_{23})^{(\tau_K+r)} \cdot \lambda^{-4\tau_K-2r}\cdot gr(\rho_{23})^{-(\tau_K+r)}\\
   &= \lambda^{-4\tau_K-2r} \cdot u^{-(vx+1)(\tau_K+r)}= (-4\tau_K-2r;0,0;(\tau_K+r)(vx+1))
\end{align*}

Now comparing $gr(a\otimes \xi_2)$ in Equation \ref{eq3} with $gr(b'_1 \otimes \lambda^j_{k_j})$, we see that Equality \ref{eq} only holds if 
\begin{align*}
    &4\tau_K+2r= (\tau_K+r + \tau_K)(vx+1)\\
    &\Rightarrow (2\tau_K+r)(vx-1) = 0\\ 
   &\Rightarrow 2\tau_K+r = 0 \Rightarrow 4\tau_K = 1
\end{align*}
Which is a contradiction, since $\tau_K$ is an integer.
\end{subsection}


\begin{subsection}{\underline{\textbf{Case 1.2: $\varepsilon(K)=1, t = 2 \tau(K) - r=0$}}}

In this case we consider the element $\xi_2$. Recall this element is the generator of the horizontal homology and $\xi_2 \in Im(\partial^{vert})$ since $\varepsilon(K) = 1$ (recall from discussions at the beginning of Section \ref{case1}). 

For this case, we consider $a \otimes \xi_0$ and $a \otimes \xi_2$. For any $\xi_{2s}$ in the vertically simplified basis of $\cfkminus$ such that $A(\xi_{2s}) = A(\xi') := a$ and let $M(\xi_{2s}) = M(\xi') := m$, (see the proof of \cite[Lemma 2.1]{hom11}). Hence

\begin{align*}
    &gr(a \otimes \xi_{2s}) = gr(\xi_{2s}) = \lambda^{m-2a} \cdot gr(\rho_{23})^{-a} \\
& \sim u^{-a(vx+1)} \cdot gr(\rho_{23})^a \cdot  \lambda^{m-2a} \cdot gr(\rho_{23})^{-a} \\
&= \lambda^{m-2a} \cdot u^{-a(vx+1)} \\
&= (m-2a;0,0; a(vx+1))
\end{align*}

Hence,

\[gr(a\otimes \xi_0) = (-2\tau_K;0,0;\tau_K(vx+1))\]
\[gr(a \otimes \xi_2) = (0;0,0;-\tau_K(vx+1))\]

Now comparing $gr(a \otimes \xi_0)$ and $gr(a \otimes \xi_2)$, we get that Equation \ref{eq} can only happen when $\tau_K = 0$. This case will be taken care of in Lemma \ref{lemma13} and the calculations right after, see Remark \ref{tau0}.

\bigskip

The following Lemma completes the proof for this case in hand.

\begin{lemma}
\label{lemma14}
  If $t = 0$,
then $a \otimes \xi_{2s}$ is non-trivial in the knot Floer homology of the cable.
\end{lemma}

\begin{proof} We check that there is no possibility of any $\mathcal{A}_{\infty}$ relations in $\cfahat(D^2 \times S^1, T_{p,q})$ involving $a$, which has exact same coefficients from the vertical, horizontal and unstable chain of $\cfdhat(X_K,r)$.

To do that, first we check the vertical chain in $\cfdhat(X_K,r)$ and the coefficient maps from Theorem \ref{lotmain}. As we mentioned in the Remark \ref{remark}, there is no disk connecting $a$ and starting with coefficient $\rho_{123}$, since in the algorithm, the $\beta$ curve always lie along the standard form of the knot, and only changes its direction around $w$. Hence there is no $\mathcal{A}_{\infty}$ relation such as $m_k(a,\rho_{123},\cdots) = c$, for some $c$.

Now, we check the horizontal chain in $\cfdhat(X_K,r)$ and the coefficient maps from Theorem \ref{lotmain}. We can see that the only relation involving $\rho_2, \rho_3$ and $\rho_{23}$ and involving $a$ is $m_k(a, \rho_3, \underbrace{\rho_{23}, \cdots, \rho_{23}}_{n},\rho_2) = U^{n+1} \cdot a$, as the said relation indicates a positive multiple of the primitive periodic domain. Thus the element $a \otimes \xi_{2s}$ is non-zero in the knot Floer homology of the cable.

Lastly we check the unstable chain in $\cfdhat(X_K,r)$ and coefficient maps there when $t = 0$. There is a single boundary component concerning $\xi_0$ and $\eta_0$ with coefficient $D_{12}$. Even if $\xi_{2s}$ is $\xi_0$ or $\eta_0$, $a \otimes \xi_{2s}$ is still non-trivial in the knot Floer homology of the cable. This is because 
there is no $\mathcal{A}_{\infty}$ relation such as $m_2(c,\rho_{12}) = a$ or $m_2(a,\rho_{12}) = c$, for any $c$, in other words, there is no Whitney disk, connecting $a$ and some other intersection point $c$, without the basepoints, covering $\rho_{12}$.

Hence our claim is proved.

\end{proof}
\end{subsection}

\begin{subsection}{\underline{\textbf{Case 2: $\varepsilon(K)=0$}}}
\label{case2}

For this case we consider an element $\xi \in \hfkhat(K) $, which lies in the lowest Alexander grading i.e. $A(\xi) = -g$. Since $\varepsilon(K)= 0$ implies $\tau_K = 0$, there is an element $\xi' \in \hfkhat(K)$ such that $\partial^{vert} \xi' = \xi$.  Then there is some $\xi_{2s}$ in the vertically simplified basis of $\cfkminus$ such that $A(\xi_{2s}) = A(\xi) = -g$ and let $M(\xi_{2s}) = M(\xi) := m$, (see the Proof of \cite[Lemma 2.1]{hom11}). Then 
\[gr(a \otimes \xi_{2s}) = (m+2g;0,0;-g(vx+1))\] 

Now we consider a \textit{dual} element $\tilde{\xi}$ of $\xi$ in the vertically simplified basis of the knot i.e. $A(\tilde{\xi}) = g, M(\tilde{\xi}) = m + 2g$. Since $\tau_K = 0$, there exists $\xi'' \in \widehat{HFK}$ such that $\partial^{vert} (\tilde{\xi}) = \xi''$. Hence there exist $\kappa \in \widehat{CFD}(X_K)$ such that 
$\tilde{\xi} \xrightarrow{D_1} \kappa $. Hence
\[gr(\kappa) = \lambda^{-1} \cdot gr(\rho_1)^{-1} \cdot gr(\tilde{\xi}) = \lambda^{-1} \cdot gr(\rho_1)^{-1} \cdot \lambda^{m} \cdot gr(\rho_{23})^{-g} \]

\begin{align*}
    gr(b_1 \otimes \kappa) = \lambda \cdot u ^{vx} \cdot gr(\rho_1) \cdot \lambda^{-1} \cdot gr(\rho_1)^{-1} \cdot \lambda^{m} \cdot gr(\rho_{23})^{-g} \\
    = \lambda^{m} \cdot u^{vx} \cdot gr(\rho_{23})^{-g}\\
    \sim u^{-g(vx+1)}\cdot gr(\rho_{23})^g \cdot \lambda^{m} \cdot u^{vx} \cdot gr(\rho_{23})^{-g}\\
    = \lambda^{m} \cdot u^{-vgx+vx-g} \\
    = (m;0,0;vgx-vx+g)
\end{align*}

\begin{lemma}
\label{lemma13}
If $\tau_K = 0$, $a \otimes \xi_{2s}$ and $b_1 \otimes \kappa$ are non-trivial in the knot Floer homology of the cable $K_{p,q}$, if $A(\xi_{2s}) \neq 0$, where $\kappa$ is as described above.
\end{lemma}
\noindent Note that the above lemma holds independent of $t$.
\begin{proof}
  Since $\tau_K = 0$, $\xi_{2s}$ is not a generator of either vertical or horizontal homology. Thus we just have to check that there is no possibility of any $\mathcal{A}_{\infty}$ relations in $\cfahat(D^2 \times S^1, T_{p,q})$ involving $a$, which has the same coefficients from the vertical and horizontal chain of $\cfdhat(X_K,r)$.

To do this, first we check the vertical chain in $\cfdhat(X_K,r)$ and the coefficient maps from Theorem \ref{lotmain}. We said in the proof of Lemma 12, there is no disk connecting $a$ and starting with coefficient $\rho_{123}$, since in the algorithm, the $\beta$ curve always lie along the standard form of the knot, and only changes its direction around $w$. Hence there is no $\mathcal{A}_{\infty}$ relation such as $m_k(a,\rho_{123},\cdots) = c$, for some $c$.

Now, we check the vertical chain in $\cfdhat(X_K,r)$ and the coefficient maps from Theorem \ref{lotmain}. We can see that the only relation involving $\rho_2, \rho_3$ and $ \rho_{23}$ and involving $a$ is $m_k(a, \rho_3, \rho_{23}, \cdots, \rho_2) = U^{l} \cdot a$, where $l$ is non-zero, as the said relation indicates a positive multiple of the primitive periodic domain. Thus the element $a \otimes \xi_{2s}$ is non-zero in $\hfkhat(K_{p,q})$.

Now we note that $b_1\otimes \kappa$ is a non-zero generator in the homology. This is because, as mentioned above, any Whitney disk, containing no $z, w-$ basepoints, connecting $b_1$ has a $\rho_3$ on its boundary. While there is no such algebra relations concerning $\kappa$.

Hence our claim is proved.
\end{proof}

Now comparing $gr(a \otimes \xi_{2s})$ and $gr(b_1 \otimes \kappa)$ we get that the equality \ref{eq} can only occur when \[ vx(2g-1) = 0\] which is not true since $vx >0, g>0$.

\begin{remark}
\label{tau0}
Note that to prove this case, we only used the fact that $\tau_K = 0$, which is a weaker condition that $\varepsilon(K) = 0$. The latter implies that former, but not the other way around. Hence the case can be though of as proving that $K_{p,q}$ is non Floer-thin when $\tau_K = 0, q>1$ and $K$ is a non-trivial knot.
    
\end{remark}

\end{subsection}

\begin{subsection}{\underline{\textbf{Case 3: $\varepsilon(K)=-1$}}} 
\label{case 3}

For $\varepsilon(K)=-1$, \cite[Section 4.2]{hom11} implies that 
\begin{itemize}
    \item $\xi_0$ is the generator of the vertical homology.
    \item $\partial^{vert}(\xi_1) = \xi_2$.
    \item $\xi_1$ is the generator of the horizontal homology i.e. $\eta_0 = U^{k} \cdot \xi_1$, for some $k$.
\end{itemize}

Notice that the symmetry of $\hfkhat(S^3,K)$ in terms of $i,j$ grading implies that since $\partial^{vert}(U^{k}\cdot \xi_1) = \partial^{vert}(\eta_0) \neq 0$, we have that $\partial^{hor}(\xi_0) \neq 0$. Let $\eta_{2j-1}:= U^{m} \cdot \xi_0 $, for some $m, j$. Thus there exists $\lambda^1_j$ such that $\xi_0= \eta_{2j-1} \xrightarrow{D_3} \lambda_1^j$.

\begin{lemma} \label{epsilon-1}
For $\varepsilon(K) = -1$, the following generators are non-trivial in $\hfkhat(S^3, K_{p,q})$:
\begin{itemize}
    \item $b_1 \otimes \kappa_1^1, a\otimes \xi_2$, for any value of $t$.
    \item $a \otimes \xi_0$, for $t \leq 0$.
    \item $b_1\otimes \mu_1$, for $t >1$.
    \item $b_1' \otimes \lambda^j_1$, for $t=1$.
\end{itemize}
where $\lambda^j_1$ is as mentioned above.
    
\end{lemma}

\begin{proof}
    Recall that there is no Whitney disk connecting $a, b_1, b'_1$ in the doubly pointed Heegaard diagram $\mathcal{H}(p,q)$ can have a $\rho_{123}$ and without any basepoints inside. This is because for any Whitney disk without the base points connecting $a$, $\Tilde{\beta}$ always has negative slope in $\mathcal{H}(p,q)$. This eliminates the possibility of having a box tensor pairing in the boundary with $\mathcal{A}_{\infty}$ relation $D_{123}$.

    To check that both $b_1 \otimes \kappa_1^1, a\otimes \xi_2$ are non-zero generators in the homology of the cable, we need to check only the vertical and horizontal chain in $\cfdhat(X_K)$. 
    
    Note that there is no algebra relations of the form $m_{i+2}(c,\underbrace{\cdots}_i,\rho_1)=b_1$ and \\ $m_{k+1}(c, \underbrace{\rho_{23}, \cdots, \rho_{23}}_k)=b_1$, since any Whintey disk connecting $b_1$ without any $w,z-$base points contains $\rho_2$ on its boundary. Also note that any Whitney disk connecting $a$ without any $z,w-$basepoints contains $\rho_1$ on its boundary. These prove that $b_1 \otimes \kappa_1^1, a\otimes \xi_2$ are non-trivial elements in the homology.

    For $t\leq 0$, the above mentioned argument about $a$ shows that $a \otimes \xi_0$ is non-zero in the homology.

    Similarly for $t>1$, the above mentioned argument about $b_1$ shows that $b_1\otimes \mu_1$ is non-zero in the homology. 

    For $t=1$, $b'_1 \otimes \lambda_1^j$ is non-zero in the homology, since any Whitney disk connecting $b'_1$ without any basepoints always has a $\rho_1$ on its boundary. Indeed, there is no algebra relation connecting $\lambda^j_1$ which has a $\rho_1$ differential. 

\end{proof}

Notice that $\eta_0 = \xi_1 \xrightarrow{D_1} \kappa^1_1$ implies 
\[gr(\kappa_1^1) = \lambda^{-1} \cdot gr(\rho_1)^{-1} \cdot gr(\xi_1) =  \lambda^{-1} \cdot gr(\rho_1)^{-1} \cdot \lambda^{0} \cdot gr(\rho_{23})^{\tau_K}\]

Then we proceed to calculate the grading of the first two generators:
\begin{align*}
gr(b_1 \otimes \kappa^1_1) &= gr(b_1) \cdot gr(\kappa^1_1) \\ &= 
  \lambda^{-1} \cdot u^{vx} \cdot gr(\rho_1) \cdot \lambda^{-1} \cdot gr(\rho_1)^{-1}\cdot gr(\rho_{23})^{\tau_K} \\
  &\sim \lambda^{-2} \cdot u^{vx} \cdot gr(\rho_{23})^{\tau_K} \cdot u^{\tau_K(vx+1)} \cdot gr(\rho_{23})^{-\tau_K} \\
  &= \lambda^{-2} \cdot u^{vx(\tau_K+1))+\tau_K} \\
  & = (-2;0,0; - (vx(\tau_K+1)+\tau_K))
\end{align*}

\begin{align*}
    gr(a \otimes \xi_0) = gr(a) \cdot gr(\xi_0) = gr(\xi_0) &= \lambda^{-2\tau_K} \cdot gr(\rho_{23})^{-\tau_K} \\
    &\sim \lambda^{-2\tau_K} \cdot gr(\rho_{23})^{-\tau_K} \cdot u^{-\tau_K(vx+1)} \cdot gr(\rho_{23})^{\tau_K} \\
    &= (-2\tau_K;0,0;\tau_K(vx+1))
\end{align*}

Thus for $t \leq 0$, we compare the grading of the above generators to get that Equation \ref{eq} can only hold when 
\[2\tau_K - 2 = 2\tau_K + 2\tau_Kvx+vx \Rightarrow vx(2\tau_K+1) = -2\]
Since $vx \geq 2$, the above implication only occurs when $vx=2, \tau_K=-1$ i.e. \[dim(\hfkhat(S^3, K)) = 2vx-1=3\] which implies that the companion $K$ is the left hand trefoil, or $T_{2,-3}$ \cite{balvel}.

\begin{figure}[ht]
    \centering
    \begin{tikzcd}
	{\xi_1} & {\mu_{|t|}} \\
	{\kappa^1_1} & {} & {\mu_1} \\
	{\xi_2} & {\lambda^1_1} & {\xi_0}
	\arrow["{D_{123}}", from=3-1, to=2-1]
	\arrow["{D_3}", from=3-2, to=3-1]
	\arrow["{D_2}", from=3-3, to=3-2]
	\arrow["{D_1}"', from=1-1, to=2-1]
	\arrow["{D_{123}}"', from=3-3, to=2-3]
	\arrow["{D_2}"', from=1-2, to=1-1]
	\arrow["{D_{23}}"', dashed, from=2-3, to=1-2]
\end{tikzcd}
   \caption{$\cfdhat(X_K)$ where $K = T_{2,-3}$. Note that the dotted arrow indicate there are $|t|$ many generators with the boundary coefficient $D_{23}$ in between them}
\end{figure}

\[gr(a\otimes \xi_2) = gr(a) \cdot gr(\xi_2) = \lambda \cdot gr(\rho_{23})^{0} = (1;0,0;0)\] Comparing $gr(a\otimes \xi_0), gr(a \otimes \xi_2)$, we get that Equation \ref{eq} can happen if \[vx+1=1 \Rightarrow vx=0\] which is a contradiction since $vx \geq 2$.

Then we deal with the subcase of $t >1$. 

\begin{align*}
    gr(b_1 \otimes \mu_1) = (-2\tau_K; 0,0; \tau_Kvx + \tau_K - vx)
\end{align*}

The last equation is obtained by using the grading calculations made for the case of $\varepsilon(K) = 1, t > 1$, since the grading calculation did not use any of $\varepsilon(K)=1$ properties.

Comparing $gr(b_1 \otimes \kappa^1_1), gr(b_1 \otimes \mu_1)$, we get that Equation \ref{eq} when \[2\tau_K - 2 = 2\tau_Kvx+ 2\tau_K \Rightarrow \tau_Kvx = -1\]

which is a contradiction since $vx \geq 2$ and $\tau_K$ is an integer.

Towards completion, we deal with the $t=1$ case. 

 Since $\xi_0= \eta_{2j-1} \xrightarrow{D_3} \lambda_1^j$. Thus 
\begin{align*}
 gr(\lambda^j_1) = \lambda^{-1}\cdot gr(\rho_3)^{-1} \cdot gr(\xi_0) &= \lambda^{-1}\cdot gr(\rho_3)^{-1} \cdot \lambda^{-2\tau_K}\cdot gr(\rho_{23})^{-\tau_K} \\
 &=\lambda^{-1-2\tau_K} \cdot gr(\rho_3)^{-1} \cdot gr(\rho_{23})^{-\tau_K}   
\end{align*}

\begin{align*}
    gr(b'_1 \otimes \lambda_1^j) &= gr(b'_1) \cdot gr(\lambda_1^j) \\ &= \rho_1 \cdot \lambda^{-1-2\tau_K} \cdot gr(\rho_3)^{-1} \cdot gr(\rho_{23})^{-\tau_K}\\ 
    &= \lambda^{-1-2\tau_K} \cdot gr(\rho_{12}) \cdot gr(\rho_{12})^{-1} \cdot gr(\rho_1) \cdot gr(\rho_3)^{-1} \cdot gr(\rho_{23})^{-\tau_K}\\
    &= \lambda^{-1-2\tau_K} \cdot gr(\rho_{12}) \cdot gr(\rho_2)^{-1} \cdot gr(\rho_1)^{-1} \cdot gr(\rho_1) \cdot gr(\rho_3)^{-1} \cdot gr(\rho_{23})^{-\tau_K}\\
    &=  \lambda^{-2\tau_K} \cdot gr(\rho_{12}) \cdot gr(\rho_{23})^{-1-\tau_K} \\
    &\sim \lambda^{-2\tau_K} \cdot gr(\rho_{12}) \cdot gr(\rho_{23})^{-1-\tau_K}  \lambda^{-1} \cdot gr(\rho_{23})^{-r} \cdot gr(\rho_{12})^{-1} \\
    &= \lambda^{-2\tau_K-1} \cdot (gr(\rho_{12}) \cdot gr(\rho_{23})^{-1} \cdot gr(\rho_{12})^{-1})^{r+\tau_k+1} \\
    &= \lambda^{-2\tau_K-1} \cdot (\lambda^{-2} \cdot gr(\rho_{23})^{-1})^{r+\tau_K+1}   \text{[by Equation \ref{eq4}]} \\
    &= \lambda^{-4\tau_K-2r-3}\cdot gr(\rho_{23})^{-(r+\tau_K+1)} \\
    &\sim \lambda^{-4\tau_K-2r-3}\cdot u^{-(vx+1)(r+\tau_K+1)}\\
    &= (-4\tau_K-2r-3;0,0;(vx+1)(r+\tau_K+1))
    \end{align*}

Comparing $gr(b'_1 \otimes \lambda_1^j), gr(b_1 \otimes \kappa^1_1)$ we see that Equation \ref{eq} occurs when
\begin{align*}
    4\tau_K+2r+1 = (vx+1)(r+2\tau_K+1)+vx &\Rightarrow (4\tau_K+1)(vx-1)=-2\\
    &\Rightarrow 4\tau_K+1 = -1 \hspace{.2in}  \text{or}, \hspace{.2in}  4\tau_K+1 = -2
\end{align*}

Note that in the above calculation we used the fact that $t = 2\tau_K - r = 1$. The last equations occurs since $vx-1 \geq 1$. The last two implications are contradictions as $\tau_K$ is an integer.
\end{subsection}

\begin{subsection}{\underline{$K_{p,1}$ is not Floer-thin}}

For the final case where the pattern knot is $T_{p,1}$, we will use the $\mathcal{A}_{\infty}$ relations described in \cite{hom11} and \cite{petkova13} and follow the strategy along the lines in that for $q \neq \pm 1$ case. For the convenience of the reader, we list the $\mathcal{A}_{\infty}$ relations from \cite{petkova13}, corresponding to Figure \ref{fig4}.

\begin{align*}
    m_1(b_k) &= b_{2p-k-1}  \hspace{.2in} \Delta_U= p-k  \hspace{.2in}  1\leq k \leq p-1 \\ 
    m_{3+i}(b_k, \rho_2, \underbrace{\rho_{12}, \cdots, \rho_{12}}_{i}, \rho_1) &= b_{k+i+1}  \hspace{.2in} \Delta_U= i+1  \hspace{.2in} \substack{1\leq k \leq p-2\\ 0 \leq i \leq p-k-2} \\
    m_{3+i}(b_k, \rho_2, \underbrace{\rho_{12}, \cdots, \rho_{12}}_{i}, \rho_1) &= b_{k-i-1} \hspace{.2in} \Delta_U= 0  \hspace{.2in} \substack{ p+1 \leq k \leq 2p-2\\ 0 \leq i \leq k-1-p} \\
  m_{2+i}(a,\underbrace{\rho_{12},\cdots, \rho_{12}}_{i}, \rho_1) &= b_{2p-i-2} \hspace{.2in} \Delta_U= 0 \hspace{.2in} 0\leq i \leq p-2 \\ 
   m_{4+i+j}(a, \rho_3, \underbrace{\rho_{23}, \cdots, \rho_{23}}_{j}, \rho_2, \underbrace{\rho_{12},\cdots,\rho_{12}}_{i}, \rho_1) &= b_{i+1} \hspace{.2in} \Delta_U= pj+i+1  \hspace{.2in} \substack{0\leq i \leq p-2 \\ 0\leq j } \\
   m_{3+j} (a, \rho_3, \underbrace{\rho_{23}, \cdots, \rho_{23}}_j, \rho_2) &= a\hspace{.2in} \Delta_U= p(j+1)  \hspace{.2in} j\geq 0
    \end{align*}
    
    From the last relation, plugging $j=0$ we get that in this case, the grading set for $CFA^{-}$ is isomorphic to $gr(\rho_{23}) \cdot u^{-p} / \Tilde{G}$.

As before, we divide our search for $(p,1)$ cables dividing into three subcases, as the value of $\varepsilon(K)$ being $0,-1,1$. 

For $\varepsilon(K) = 0$, we can proceed exactly as in case for $(p,q)$ cabling and $t \neq 0, \varepsilon(K) = 0$. One only needs to notice that $a \otimes \xi_{2s}$ is also a non-zero generator for $(p,1)$ cables as well. This is since there the only $\mathcal{A}_{\infty}$ relations regarding $a$ (without any $U$ power) are
\[m_{2+i}(a,\underbrace{\rho_{12},\cdots, \rho_{12}}_{i}, \rho_1) = b_{2p-i-2}\]
and there are no such relations concerning $\xi_{2s}$ in in $\widehat{CFD}(X_K)$. 

Also the gradings of the mentioned elements in the knot Floer homology of $(p,1)$ cables are:
\[gr(a \otimes \xi_{2s}) = (m+2g;0,0;-gp)\] \[gr(b_1 \otimes \kappa) = (m;0,0;gp-p+1)\]

Which does not satisfy Equation \ref{eq}.

When $\varepsilon(K) = 1$ and $t \neq 0$, the proof for the case $T_{p,q}, q>1$ holds true for $q=1$ as well (by replacing $vx+1$ with $p$ in the grading calculations for the $(p,q)$ cable, and by the fact that $p \geq 2$).

For $\varepsilon(K) = 1$ and $t=0$, we consider $a \otimes \xi_0$ and $a \otimes \xi_2$ (recall from discussions at the beginning of Section \ref{case1}). Both are non-zero elements in the knot Floer homology of the $(p,1)$ cable. This follows from the fact that there is no $\mathcal{A}_{\infty}$ relations in $\widehat{CFD}(X_K)$ containing $D_{12}, D_{12}, \cdots, D_1$ corresponding to $\xi_0$ and $\xi_2$, when $t=0$.  

Now \[gr(a\otimes \xi_0) = gr(\xi_0) = \lambda^{-2\tau_K} \cdot gr(\rho_{23})^{-\tau_K} = (-2\tau_K;0,0;p\tau_K)\] 
\[gr(a\otimes \xi_2) = gr(\xi_2) = gr(\rho_{23})^{\tau_K} = (0;0,0;-p\tau_K)\]

Which satisfies equality \ref{eq} only if 
\[(p-1)\tau_K = 0 \Rightarrow \tau_K =0\]

For $\tau_K = 0$, we can proceed the same way as we did for the case $\epsilon(K) = 0$, as in that previous proof we only used the fact that $\epsilon(K) = 0 \Rightarrow \tau_K = 0$.

For the case when the pattern is $(p,1)$ and when $\varepsilon(K)= -1$, our strategy will be similar to the previous $\varepsilon(K)=-1$ for non-trivial torus knot patterns. 

First we note that from the $\mathcal{A}_{\infty}$ relations, it is easy to see that any element of the form $b_1\otimes \cdot$ is a non-zero generator in $\hfkhat(S^3, K_{p,pr+1})$, since any relations corresponding to the intersection $b_1$ has a non-zero $\Delta_U$ filtration. As before, we consider $b_1 \otimes \kappa^1_1$, for any value of $t$. For $t \neq 0$, we consider $b_1 \otimes \mu_1$. Note that we can reuse the grading calculation of for $b_1 \otimes \kappa^1_1, b_1 \otimes \mu_1$ with only replacing $vx+1$ by $p$, which is the number of $w-$basepoints in the periodic domain connecting $a$. Thus

\[gr(b_1 \otimes \kappa^1_1) = (-2;0,0;-(p\tau_K+p-1))\]
\[gr(b_1 \otimes \mu_1)= (-2\tau_K;0,0; p\tau_K-p+1)\]

Comparing the above grading we see that Equality \ref{eq} can only occur if 
\[\tau_K-1 = p\tau_K \Rightarrow (p-1)\tau_K=-1 \Rightarrow p=2, \tau_K=-1\]
where the last equality follows from the fact that $p \geq 2$.

For the remaining cases when $t \neq 0$, we can consider $b'_1 \otimes \lambda^j_{1}$, as earlier, but here we notice that $b'_1$ is $b_{2p-2}$ in Figure \ref{fig4}. Note that it is also a non-zero element in the homology as the only $\mathcal{A}_{\infty}$ relations with $\Delta_U=0$ involving $b_{2p-2}$ does not have $\rho_3$ or $\rho_{23}$ coefficients. Now we calculate the grading of this generator using the calculation from before and replacing $vx+1$ by $p$ and putting $\tau_K=-1, p=2$. 

\[gr(b_{2p-2}\otimes \lambda^j_1) = (1-2r;0,0;2r)\]

Comparing this with $gr(b_1 \otimes \kappa^1_1)$ we get that Equality \ref{eq} can never occur. 

For the remaining case of $t=0$, we consider, as before, the element $a\otimes \xi_0$. Notice that this is also non-zero in the homology as there is no $\mathcal{A}_{\infty}$ relations concerning $a$ with $\Delta_U=0$. By the calculations from before,

\[gr(a\otimes \xi_0) = (-2\tau_K;0,0;p\tau_K)\]

Comparing $gr(a\otimes \xi_0), gr(b_1 \otimes \kappa^1_1)$, we see that equality \ref{eq} can only occur if 

\[(1-p)(2\tau_K+1) = 2\] since $p \geq 2$, it follows that the above equality can only occur if $p=3, \tau_K=-1$. Note that this is the exact situation we encountered for the pattern $T_{p,q}, q>1$ and $\varepsilon(K) = -1$. We deal with this case similarly and consider $gr(a\otimes \xi_2)$. Then by comparing $gr(a\otimes \xi_0)$, $gr(a\otimes \xi_2)$ we arrive at a contradiction, as before.

\end{subsection}

 This completes our proof. \qedsymbol
  \vspace{5mm}

\bibliographystyle{amsplain}

\end{document}